\documentclass[12pt,a4paper]{amsart}

\usepackage[abbrev]{amsrefs}

\usepackage{latexsym,subfigure}
\usepackage{amsmath,amsthm,amsfonts,amssymb,graphicx,epsfig,latexsym,color}

\def\S{{\Sigma}}

\def\d{{\partial}}

\date{\today}

\usepackage{color}

\theoremstyle{plain}
  \newtheorem{thm}{Theorem}[section]
  \newtheorem{lem}[thm]{Lemma}
  
  \newtheorem{cor}[thm]{Corollary}
  \newtheorem{prop}[thm]{Proposition}

\theoremstyle{definition}

  \newtheorem{ex}[thm]{Example}

\theoremstyle{remark}
  \newtheorem{rem}[thm]{Remark}

\numberwithin{equation}{section}

\DeclareMathOperator{\vol}{vol}

\DeclareMathOperator{\isom}{Isom}
\DeclareMathOperator{\Wh}{Wh}

\newcommand{\field}[1]{\mathbb{#1}}
\newcommand{\tF}{{\tilde F}}

\newcommand{\R}{\field{R}}

\newcommand{\Z}{\field{Z}}

\begin{document}

\title[]
{Graph manifolds as ends of negatively curved Riemannian manifolds}

\thanks{The authors are partially supported by a Grant-in-Aid
for Scientific Research from the Japan Society for the Promotion of Science,
15H05739, 26400060}

\begin{abstract}
Let $M$ be a graph manifold such that each piece
of its JSJ decomposition has the $\Bbb H^2 \times \R$
geometry. Assume that the pieces are glued by isometries. 
Then,  there exists a complete Riemannian metric on $\R \times M$
which is an ``eventually warped cusp metric'' with the sectional curvature
$K$ satisfying $-1 \le K <0$.

A theorem by Ontaneda then implies that $M$ appears as an end 
of a 4-dimensional, complete, non-compact Riemannian manifold
of finite volume with sectional curvature $K$ satisfying $-1 \le K <0$.

\end{abstract}

\author{Koji Fujiwara}
\author{Takashi Shioya}

\address{Mathematical Institute, Tohoku University, Sendai 980-8578, JAPAN}
\address{Department of Mathematics, Kyoto University, Kyoto 606-8502, JAPAN}
\email{shioya@math.tohoku.ac.jp}
\email{kfujiwara@math.kyoto-u.ac.jp}

\date{\today}

\dedicatory{Dedicated to Professor Kenji Fukaya on the occasion of his sixtieth birthday}


\maketitle

\section{Introduction and main theorem}
\label{sec:intro}

\subsection{Ends of manifolds of negative curvature}

 If a non-compact manifold $N$ is the interior
 of some compact manifold with boundary, 
 then each boundary component is called
 an {\it end} of $N$.
 Let $N$ be a complete, non-compact Riemannian manifold
of finite volume such that the sectional curvature $K$
satisfies $-1 \le K <0$.
It is known by Gromov-Schroeder \cite{BGS} that $N$ is diffeomorphic
to the interior of a compact manifold, $\bar N$,
with boundary, $\partial \bar N$.
$\partial \bar N$ has finitely many components, and
each component is an end of $N$.

It is a wide open question to decide which manifolds, $M$, can appear as ends 
of such Riemannian manifolds. 
An end  is a closed manifold and one general obstruction 
by Gromov \cite[0.5]{G} is that  the simplicial volume of $\partial \bar N$, hence, of each end 
is zero. 
Also, the $\ell^2$-betti number and the Euler characteristic
vanish (see \cite[Corollary 15.7]{igor1}).
It is a theorem by Avramidi-Nguyen Phan \cite[Corollary 5]{AP} that 
if the dimension of an end we concern 
is at most 4, then it is aspherical. 
In this paper we address the question: which aspherical 
manifolds can appear as such ends? 

For example, an $n$-dimensional torus appears as an end
of an $(n+1)$-dimensional hyperbolic manifold
of finite volume. 
Other examples of ends are 
circle bundles over some hyperbolic manifolds of various 
dimension, \cite{fujiwara} (cf. \cite{igor2}, \cite{M} for the compex
hyperbolic versions).
In contrast to tori, such bundles will not be ends
of any complete, non-compact Riemannian manifold of finite
volume such that $-1 \le K \le -c <0$ for some $c>0$, 
since under this curvature assumption, the fundamental group of 
an end has to be virtually nilpotent.  

In dimension 2, if a closed, aspherical  manifold has zero simplicial volume, then it is a torus or a Klein bottle, and it appears as an end, for exmaple in the figure eight knot complement and in the Gieseking manifold.
In dimension 3, any closed aspherical manifold $M$ is 
irreducible, has an infinite fundamental group
and its universal cover is $\R^3$, cf.\cite{Lu}.
If the simplicial volume of such $M$ is 0 then 
it is a graph manifold. 

A {\em graph manifold} is an aspherical,  
closed 3-manifold whose JSJ decomposition along embedded incompressible tori/Klein bottles contains only 
Seifert fibred spaces. Abresch-Schroeder \cite{AS} proved 
certain graph manifolds appear as ends.
Our theorem will provide a  large class
of graph manifolds that appear as ends, 
and their examples are 
contained in our class
(but for such manifold $M$, their manifold $N$ that has $M$ as an end is different from ours).
Also,  every 3-dimensional sol-manifold appears as an end, \cite{phan}.

Other known examples are infranilmanifolds, \cite{onta}, \cite{BeK}.
See also \cite{igor} for an axiomatic construction
from known examples. 

\subsection{Eventually warped product cusp metric}
In this paper we show that a family of (3-dimensional) graph
manifolds occurs as ends of complete, non-compact, Riemannian manifolds
of finite volume whose sectional curvature $K$ satisfies $-1 \le K <0$.

To explain our strategy, we recall the following groundbreaking theorem
by Ontaneda, \cite{onta}. 
If a (not necessarily connected) manifold $B$ is diffeomorphic to the boundary
of a connected, smooth, compact manifold $N$, then we say that $B$ {\it bounds}
$N$. We sometimes say $B$ {\it bounds} without specifying $N$.


\begin{thm}[Ontaneda]\label{thm.ontaneda}
Let $B$ be a closed manifold such that 
either $\dim B \le 4$ or the Whitehead group $\Wh(B)$ vanishes. 

Assume 
 $\Bbb R \times B$ admits a complete Riemannian metric
$g$ such that 
\begin{enumerate}
\item there exists a constant $C<0$ with 
the sectional curvature of $g$ satisfying  $C \le K <0$,
\item $(-\infty, 0] \times B$ has finite $g$-volume,
\item there is $D >0$ such that on $[D,\infty) \times B$,
the metric $g$ is of the form $g=dr^2 + e^{2r}g_B$
for some Riemannian metric $g_B$ on $B$.
\end{enumerate}
Then $B \sqcup B$ bounds a manifold whose interior admits a complete Riemannian 
metric of finite volume and the sectional curvature in $[-1, 0)$.
\end{thm}

A metric on $\Bbb R \times B$ that satisfies
the condition (2) is called a {\it cusp} metric and an
{\it eventually warped {\rm(}cusp{\rm)} metric} if it satisfies (3).
This theorem is stated only implicitly in \cite{onta}
(see \cite{igor}, where the result is quoted in this form),
since it is an intermediate claim, but a detailed argument
is given. The actual value of $C$ is not important
and we can take $C=-1$ by rescaling $g$.

We will show that for a manifold $B$ in  certain families
there exists a Riemannian metric on $\R \times B$ that 
is an eventually warped cusp metric with $C \le K <0$ for some 
$C<0$.
Then Theorem \ref{thm.ontaneda} implies that $B$ is an end of a manifold of
negative curvature. 

This argument appears in \cite{onta}
for the infranil manifolds (The existence of a desired metric is known
by \cite{BK}) then also is used in \cite{igor} and  \cite{phan}
to construct other examples of ends.

\subsection{Graph manifolds and flip manifolds}
To illustrate the first family we handle, let $W$ be a 3-dimensional manifold
which is diffeomorphic to $\Sigma \times S^1$, where 
$\Sigma$ is a compact surface with non-empty boundary and $S^1$ is a circle. 
Each boundary component of $W$ is a torus, $S^1 \times S^1$,
where the first factor is a boundary component of $\Sigma$.
We put  an orientation to each factor. 
We call such $W$ a {\it piece}, and $\Sigma$ the {\it base} surface of $W$.
We construct  a {\em closed}, connected, 3-dimensional manifold $M$, 
which is a graph manifold,  from a finite collection of pieces
by gluing a pair of boundary tori by a diffeomorphism, a {\em gluing map}.
There are two special maps for gluing: the trivial map mapping 
the first factor to the first factor and the second one to the second;
the {\em flip map} interchanging the first and second factors. We preserve the orientation of the factor. 
We say $M$ is a {\it flip-manifold}, \cite{KL},
if each gluing map is either the trivial map or the flip map.

Some remarks are in order.
There are eight ways to glue a pair of boundary tori: two ways to put an orientation
on each of the two $S^1$, then a trivial map or a flip map.
If $\Sigma$ is a closed surface, then $\Sigma \times S^1$
is considered as a flip manifold made from one piece of two boundary components, 
where the gluing map is trivial.

More generally, mabye the $S^1$-fibers are not-orientable, and/or a piece is a Seifert fibered space, $S$, \cite[Section 3]{scott}. Let $s_1, \cdots, s_n$ be the singular
fibers of $S$ where the twist at $s_i$ is 
by the $q_i/p_i$ of a full twist.  $(q_i,p_i)$ 
is called the {\em orbit invariant} of $s_i$, which is a pair of co-prime
integers with $0<q_i< p_i$.
One can say that a Seifert fibered space is an $S^1$-bundle
over a base orbifold, where the singular fibers occur at the orbifold points, while 
$\Sigma \times S^1$ is a (trivial) $S^1$-bundle over the surface $\Sigma$.

A {\em generalized flip manifold} is 
a generalization of a flip manifold where
we allow Seifert fibered spaces  in addition to
products $\Sigma \times S^1$ as pieces in the definition.
Of course we only consider gluing maps that are diffeomorphisms. 

We call a base surface/orbifold $\Sigma$ {\em hyperbolic} if
we can put a complete hyperbolic (orbifold) metric of finite area 
to the interior of $\Sigma$. We denote $\Sigma^o$ the interior of $\Sigma$.
An $S^1$-bundle over $\Sigma^o$, has 
a Riemannian metric that is locally a product of the hyperbolic metric 
and $S^1$ (see \cite{scott}). In other words, it has the {\it geometry} of 
$\Bbb H^2 \times \R$, or the metric is of {\it type} $\Bbb H^2 \times \R$.
We only consider pieces of this kind in this paper. 

We truncate a small neighborhood of each cusp of $\Sigma^o$
such that the each boundary component of the 
universal cover of the truncated $\Sigma^o$ with respect to the hyperbolic metric
is a horoline in $\Bbb H^2$. 
Since $\Sigma$ is diffeomorphic to the truncated $\Sigma^o$, 
we obtain a Riemannian metric on the $S^1$-bundle over $\Sigma$
such that each boundary torus/Klein bottle is flat. We also say this metric is
of {\it type} $\Bbb H^2 \times \R$.

In this paper we say a graph manifold $M$ has a {\em geometrization} if one can put a Riemannian metric of type $\Bbb H^2 \times \R$ on all pieces such that every gluing map along the boundary tori/Klein bottle is an {\em isometry}.
Some remarks are in order. 
We do not assume that each gluing map is a trivial map or a flip map, 
see Example \ref{example.non-flip}.
A metric of type $\Bbb H^2 \times \R$ on a piece is 
not unique.
Without loss of generality, we may assume that there is a small $c>0$
such that the length of the fibers of the pieces and the length
of the boundary components of the pieces are $c$.
The resulting metric on $M$ after gluing the pieces is only $C^1$.
If $M$ has a geometrization or not depends only on the topology of $M$.
In \cite{BK}, they use the term {\em isometric geometrization}
instead of geometrization (see Remark \ref{rem.buyalo}).

We prove:
\begin{thm}\label{metric}
Let $M$ be a graph manifold such that each piece has the geometry 
of $\Bbb H^2 \times \R$. Assume $M$ has a geometrization
(i.e., the gluing is isometric). Then there is a complete Riemannian 
metric $g$ on $\Bbb R \times M$ that is an eventually warped cusp metric
with the sectional curvature $K$ satisfying $C \le K <0$
for some constant $C<0$.
\end{thm}

\begin{rem}
The metric in the above theorem is taken to be $C^{\infty}$.
This is always the case for other results in this paper too. 

\end{rem}

By rescaling the metric $g$ we may always take $C=-1$.
Theorem \ref{metric} has a generalization to high dimensional manifolds,
see Theorem \ref{high.dimension}.

Since $\dim M =3$, combining Theorem \ref{metric} and Theorem \ref{thm.ontaneda} we immediately
obtain:
\begin{cor}[Graph manifolds with a geometrization]\label{main}
Let $M$ be a graph manifold such that each piece has the geometry 
of $\Bbb H^2 \times \R$. Assume $M$ has a geometrization.

Then there exists a  $4$-dimensional, complete, non-compact
Riemannian manifold $N$ of finite volume, of the sectional curvature
$K$ satisfying $-1 \le K <0$, and with $M$ appearing as an end. More precisely, 
there is a compact subset
$C$ in  $N$ such that  $N \backslash C$ has two connected 
components, and that each component is 
diffeomorphic to $M \times (0,\infty)$.
\end{cor}

It is known that among graph manifolds $M$ whose pieces have  the geometry 
of $\Bbb H^2 \times \R$, $M$ has a geometrization 
  if and only if  it has a Riemannian metric of non-positive sectional curvature
  by  \cite{L} and \cite{LS}.
Hence we can rephrase our results as follows:
\begin{cor}[Graph manifolds of non-positive curvature]\label{cor.nonpositive}
Let $M$ be a closed graph manifold such that 
each piece has the geometry of $\Bbb H^2\times \R$.
 Assume $M$ has a Riemannian
metric of non-positive curvature. Then 
the conclusion of Theorem \ref{metric}
and Corollary \ref{main} holds. 
\end{cor}

\subsection{High dimensional graph manifolds}
There are several notions of high dimensional graph manifolds 
(cf.~\cite{sisto}) and 
one can prove a high dimensional version of Theorem \ref{metric}.
The main part of the proof of the theorem is by constructing a
suitable Riemannian metric, which is same for the high dimensional case.

Fix $n\ge 2$ and $m \ge 1$. Let $X$ be an $n$-dimensional complete, non-compact, hyperbolic manifold of finite volume
such that the cross-section of each cusp is an $(n-1)$-dimensional torus. 
Let $\Sigma \subset X$ be a compact manifold with boundary obtained by truncating a sufficiently 
small neighborhood of each cusp from $X$
so that each boundary component is a flat torus.  The interior of $\Sigma$ is 
diffeomorphic to $X$.
Take a Riemannian product $W=\Sigma \times T^m$, where 
$T^m$ is an $m$-dimensional flat torus. Each boundary component of
$W$ is an $(n+m-1)$-dimensional torus. We call $W$ a piece, and 
$\Sigma$ the base. 

Suppose a closed $(n+m)$-dimensional manifold 
$M$ is obtained from pieces with various bases by gluing pairs of boundary components of 
the pieces by diffeomorphisms,  then we call $M$ a {\em high dimensional graph manifold}.
We say $M$ has {\em geometrization} if all gluing maps are isometric
with respect to the product metric on the pieces.

\begin{thm}[High dimensional graph manifolds]\label{high.dimension}
Let $M$ be an $(n+m)$-dimensional high dimensional graph manifold.
Assume $M$ has a geometrization. 
Then $M$ carries a metric of non-positive curvature, so that 
$\Wh(M)$ vanishes. 
Also, there is a complete Riemannian 
metric $g$ on $\Bbb R \times M$ that is an eventually warped cusp metric
with the sectional curvature $K$ satisfying $C \le K <0$
for some constant $C<0$.
\end{thm}

\begin{rem}
We only consider a product metric on each piece, but we can formulate the result for locally product metrics
as in Theorem \ref{metric}.
\end{rem}

As before, it follows from Theorem \ref{thm.ontaneda}:
\begin{cor}\label{high.dimension.end}
Let $M$ be an $(n+m)$-dimensional high dimensional graph manifold.
Assume it has a geometrization. 
Then there exists an $(n+m+1)$-dimensional, complete, non-compact
Riemannian manifold $N$ of finite volume, of the sectional curvature
$K$ satisfying $-1 \le K <0$, and with $M$ appearing as an end. 
\end{cor}

\subsection{The other construction}\label{section.construction}
We discuss the other family of examples of ends. 
This family contains manifolds of various dimension, and 
in dimension 3, it contains all flip manifolds without a piece
whose base surface has genus at most 1.
Although it is not necessary we only treat the orientable case
to make the account simple and clear. 

A manifold in this family is also obtained by 
gluing pieces along their boundary, 
and each boundary component is a circle bundle over a circle bundle over a hyperbolic
manifold $N$. If $\dim N=0$,  the boundary is a torus
and we obtain graph-manifolds. 

Here is a precise description.
Let $M_i$, $i=1,2$, be $n$-dimensional closed
hyperbolic manifolds, and $N_i$ totally
geodesic, closed submanifolds of codimension 
two in $M_i$, respectively, such that $b:N_1 \to N_2$
is an isometry.
For a sufficiently small $\epsilon>0$, 
let $P_i$ be $S^1$-bundles over
$V_i=M_i\backslash N_\epsilon(N_i)$, respectively, with 
Riemannian metrics which are locally product
of the hyperbolic metric on the base manifolds and the circle. 

$\partial P_i= P_i|\partial V_i$  are flat torus-bundles over $N_i$, 
respectively. We glue $P_1, P_2$ along their boundaries
by a bundle map whose base map is the isometry $b:N_1\to N_2$ and 
on the fiber it is a diffeomorphism, for example, a trivial map or a flip map,
as in the graph manifold case. 
This gives an $(n+1)$-dimensional manifold, $W$. 
If the bundle map is an isometry, then we say it satisfies  {\em the gluing condition}
and $W$ has a {\em geometrization}. 

Then 
\begin{thm}[Theorem \ref{thm2}]\label{thm3}
Assume $W$ has a geometrization. Then 
$W$ carries a metric of non-positive curvature,
so that $\Wh(W)$ vanishes. 
Also, $\R \times W$ carries a complete Riemannian metric
that is an eventually warped cusp metric with $C \le K <0$ for
some constant $C<0$.
\end{thm}

Combining Theorem \ref{thm3} and Theorem \ref{thm.ontaneda}
we obtain:
\begin{cor}[Piecewise $S^1$-bundles]
Assume $W$ has a geometrization. Then 
$W$ appears as an end of an $(n+2)$-dimensional  
Riemmanian manifold $Z$ that is complete, non-compact,
of finite volume, with the sectional curvature
$K$ satisfying $-1 \le K <0$.
\end{cor}

\subsection{Gluing condition}
We examine the gluing condition for a geometrization in the case $n=3$ in some details,
where $n$ is the dimension of $M_i, i=1,2$.
Submanifolds $N_i$, $i=1,2$, are simple closed geodesics, and we denote them
by $\gamma_i$, respectively. By our assumption they have same length.
Let $m_i$ be the meridean curve for $\gamma_i$ in $M_i$.
$X_i$ is an $S^1$-bundle over $M_i\backslash N_\epsilon(\gamma_i)$.
We denote $\sigma_i$ the fiber circle of $X_i$.
With respect to the Riemannian metric, we can measure 
the monodromy (i.e., rotation) along the curve $\gamma_i$ for $\sigma_i$
and $m_i$, respectively. We denote them by 
$0 \le \theta(\sigma_i), \theta(m_i) <2\pi$.

We now consider a bundle map $\phi: \partial X_1
\to \partial X_2$ such that the base map is the isometry $b$
and that it is a flip map on the torus fiber: 
$$\phi(m_1)=\sigma_2, \qquad \phi(\sigma_1)=m_2$$
Then we can arrange $\phi$ to be an isometry, 
ie, the gluing condition is satisfied, 
if and only if 
\begin{equation}\label{mono.1.0}
\theta(m_1)=\theta(\sigma_2), \qquad \theta(\sigma_1)=\theta(m_2)
\end{equation}

We conclude the introduction with examples $\{(M_i,N_i)\}$ that admit circle bundles $X_i$
satisfying (\ref{mono.1.0}), which give $W$ of $\dim W=4$ by Theorem \ref{thm3}.
\begin{ex}[$S^1$-bundle with a given 
monodromy]
Take a closed, oriented, hyperbolic $3$-manifold $M$
with a simple closed (oriented) geodesic $\gamma$
which is non-trivial
in $H_1(M, \R)$.
Let $m$ be a {\it meridian curve}
around $\gamma$ in $M$. Namely, take an $\epsilon$-neighborhood
of $\gamma$ in $M$ for a small $\epsilon>0$. Its boundary is a torus, $T$.
Take a small hyperbolic disc $D$ in $M$
perpendicular to $\gamma$, then set $m= D \cap T$.
Let $\theta(m)$ be the monodromy of $m$ along $\gamma$.
Another way to define $\theta(m)$ is using the universal cover of $M$.
Lift $\gamma$ to an (oriented) infinite geodesic $\tilde \gamma$, then take the element $g \in \pi_1(M)$
such that $\tilde \gamma / g=\gamma$, where $g$ shifts
to the positive direction. Then $g$ rotates $\tilde \gamma$
by $\theta(m)$.

Set $V=M \backslash N_\epsilon(\gamma)$
with a small $\epsilon>0$.
We will construct an $S^1$-bundle over $V$, $X$, 
and glue $X$ and a copy of $X$ along their boundary
and obtain $W$.
Let $\sigma$ denote the fiber circle.
For our construction, we need to arrange $\theta(m)=\theta(\sigma)$.

For example, let $M$ be such that  all of its closed geodesics
are simple (such examples exist, \cite{CR}).
Taking a finite cover if necessary, we may assume 
that $H_1(M,\R)$ is non-trivial \cite{A}.
Let $p:\pi_1(M) \to H_1(M,\R)$
be the homomorphism obtained
from abelianization.
Take any closed (simple) geodesic $\gamma$
with $p([\gamma])\not=0$.
Then $H_1(\gamma, \R)$ injects to 
$H_1(M,\R)$.

Set $\theta_0=\theta(m)$.
Then there is a homomorphism
$h:\pi_1(M) \to S^1$ such that $h([\gamma])=\theta_0$.
Indeed we take a homomorphism
$f:\Z\to S^1$ such that $f(p([\gamma]))=\theta_0$
then set $h=f \circ p$.

Now take an $S^1$-bundle $X$ over $M$,
which is locally a Riemannian product
whose monodromy representation 
of $\pi_1(M)$ to $S^1$ is $h$.
Then $\theta(\sigma)=\theta_0$.

Now take $(M,\gamma)$ and its copy, then this pair
satisfied the condition (\ref{mono.1.0}), so that 
Theorem \ref{thm3} applies.

\end{ex}

{\bf Acknowledgement}
We would like to thank Igor Belegradek
for his interest, many valuable comments and questions. 
He also pointed out an omission of an argument in an earlier version
of the paper. 
We owe Misha Kapovich concerning Corollary \ref{cor.nonpositive}.
We are grateful to Pedro Ontaneda for explaining his work
and clarifying technical issues for us. 
We would like to thank Kenji Fukaya for many 
useful discussions.

\section{Proof of Theorem \ref{metric} and Theorem \ref{high.dimension}}
We will prove Theorem \ref{metric}.
We first treat the case where every piece in a graph manifold 
is the product of a circle and a surface,  then discuss
the general case. 

We then prove Theorem \ref{high.dimension}. The main part
of the proof overlaps with the proof of Theorem \ref{metric}, which 
is Proposition \ref{estimate}.

\subsection{Geometric idea}

We first explain our method to construct a desired metric
on $\Bbb R \times M$, where each piece of $M$
is a trivial bundle over a surface.

As the first step, it is pretty straightforward to construct a Riemannian metric
of non-positive curvature on $M$. 
We review the metric construction, cf, \cite{KL}. 
By assumption the interior of the base surface $\Sigma_i$ of each 
piece $P_i$ has a hyperbolic metric $g_0$ of finite volume. 

Choose a small constant $c>0$.
Truncate the
interior of $\Sigma_i$ with the metric $g_0$ at each cusp so that 
each boundary circle 
has constant geodesic curvature
and has
length $c$. We identify this truncated surface
with $g_0$ and  $\Sigma_i$. 

To express the idea clearly,  we first assume $M$ has a geometrization 
with respect to a product metric on each piece,
$P_i=\Sigma_i \times S^1$, namely it is a Riemannian product
$\Sigma_i \times S^1(c)$, where $S^1(c)$ is 
a circle of length $c$. The curvature
satisfies $-1 \le K \le 0$. Each boundary component of $P_i$
is $S^1(c) \times S^1(c)$, so that we 
can glue $P_i$'s along their torus boundaries 
by the prescribed  gluing maps, which are isometries by our assumption, and obtain a metric
of non-positive curvature 
on $M$. This metric has singularity along the tori
where pieces are glued, but we can smooth it 
out keeping the curvature condition $C \le K \le 0$
for some $C<0$.

In the next step we want to put a desired metric on $\R \times M$, but if we 
consider a warped product
$$ \Bbb R \times _{e^{r}} M$$
it does not work in general by the following reason. 
Since $M$ is compact, volume of 
$(-\infty, 0] \times M$ is finite.
Also, since $M$ satisfies $C< K \le 0$, 
the curvature on $ \Bbb R \times _{e^{r}} M$
satisfies $K<0$, but as $r \to -\infty$, 
the diameter of $M$ tends to $0$ and the 
curvature $K$ tends to $- \infty$, while 
$K$ tends to $-1$ as $r$ tends to $\infty$.
So, this construction violates the curvature 
bound from below for $r \to -\infty$.

In the above warped product, at each $r \in \R$, the manifold
$M$ is rescaled as $e^{r} M$.  As a remedy, 
we use another metric on $M$ for the part $r<0$ 
with $r$ small enough. Take $a>0$
such that $e^r \le c$ if $r<-a$. 
For each $r \in (-\infty,-a]$, 
truncate the initial complete
hyperbolic metric $g_0$ on the interior of $\Sigma_i$ such that the boundary 
circle has length $e^r$, which we denote
by $\Sigma_i(e^r)$. Take a Riemannian product
$\Sigma_i(e^r) \times S^1(e^r)$, which 
is the metric structure on $P_i$ at $r$.  Each boundary 
component is $S^1(e^r) \times S^1(e^r)$.
Now glue them by the given isometries
and obtain the metric on $M$ at $r$,
which we denote by $M_r$.  As before we smooth 
out near the gluing tori. 
In this way we obtain a metric on $(-\infty, -a]
\times M$, which we write as $(-\infty,-a]
\times M_r$.
Notice that volume of $M_r$ is (more or less) proportional 
to $e^r$, so that one expects volume
of $(-\infty,-a] \times M_r$ is finite. Also, we arrange that 
the curvature satisfies $ C \le  K <0$ (cf. \cite{fujiwara}).

So, we try to interpolate 
$(-\infty,-a ] \times M_r$
and $[a,\infty) \times _{e^{r}} M_0$
between $r \in [-a,a]$, where $M_0$
is $M$ with a Riemannian metric, say,
constructed in the previous paragraph
(maybe we rescale it by a constant).
Note that the metric on $M_0$
is fixed for $r \in [a,\infty)$ while
$M_r$ keeps changing for $r \in (-\infty, a]$.
Also, notice that diameter of $M_r$ tends
to $\infty$ as $r \to -\infty$.

Lastly, we address the issue that a metric on a piece
is maybe a locally Riemannian product. 
The piece is topologically a trivial $S^1$-bundle, and 
it has a locally product metric such that the fiber circles
have same length. 
The difference from the Riemannian product case is 
encoded in the monodromy representation 
of the fundamental group of the base $\Sigma$ into $S^1$, viewed as a group,
which acts on the fiber circles by rotations. 
By assumption our manifold $M$ admits a geometrization, ie, pieces are 
glued by isometries along the tori. 
In conclusion, the method we explained above will work in this generality
without any change because we only use the property that the gluing
maps are isometric.

\subsection{Metric construction}
We denote the group of isometries of ${\Bbb R}^n$ 
by $\isom(\Bbb R^n)$.
In the following we consider a product 
$$Y=\Bbb R \times \Bbb R \times \R^l \times \R^m$$ and, for example,  an element of $\isom(\Bbb R^l)$ naturally 
acts on $Y$ by an isometry that is trivial except on $\R^l$.
The Euclidean metric on $\R^l, \R^m$ are denoted by 
$d\rho^2, d\tau^2$, respectively.
We denote a flat torus of dimension $n$ by $T^n$.
For $n=1$, we may also write it as $S^1$.

The goal of the following few subsections is Proposition \ref{estimate},
which shows that a certain Riemannian metric $g$ that is 
invariant by $\isom(\R^l)$ and $\isom(\R^m)$ exists
on $Y$.
The proof of Proposition \ref{estimate} is 
by concretely constructing a metric $g$.
If tori $T^l, T^m$ are given as quotients of $\R^l, \R^m$
by isometric actions, then the metric $g$
descends to 
$$ X=\R \times \R \times T^l \times T^m,$$
which will be used later to prove theorems.

To define the metric $g$, we prepare several functions.
Pick  a $C^\infty$ function, $R$, on $\R$ such that
\[
R(r) =
\begin{cases}
  r &\text{if $r \le 1$},\\
  3 &\text{if $r \ge 5$},
\end{cases}
\qquad R' > 0 \ \text{on $(\,1,5\,)$},
\qquad -\frac12 \le R'' \le 0.
\]
We take a nonnegative $C^\infty$ function, $\lambda$, supported in $[\,-1,1\,]$
and satisfying $\int_{-1}^1 \lambda(x)\,dx = 1$,
then define the \emph{convolution product}
of $\lambda$ and a locally Lebesgue integrable function $\varphi$ on $\R$ by 
\[
\lambda * \varphi(x) := \int_{-\infty}^{\infty} \varphi(t) \lambda(x-t) dt.
\]
Note that $\lambda * \varphi$ is also defined
in the case where $\varphi$ is a finite Borel measure on $\R$.
$\lambda * \varphi$ is a $C^\infty$ function.

Since $\lambda * e^t$ satisfies $(\lambda * e^t)'' = \lambda * e^t$,
we have $\lambda * e^t = c \,e^t$, where $c := (\lambda * e^t)(0)$.
We put
\begin{gather*}
\bar{f}(s) :=
\begin{cases}
  1 &\text{if $s \le 0$},\\
  e^s &\text{if $s > 0$},
\end{cases}
\end{gather*}
then define a $C^\infty$ function by:
\[
f := \lambda * \bar{f}.
\]
By the definition,
\begin{equation} \label{eq:f}
  f(s) = 
  \begin{cases}
    1 &\text{if $s \le -1$},\\
    c \, e^s &\text{if $s \ge 1$}.
  \end{cases}
\end{equation}
We observe
\begin{align}
  \label{eq:fbp}
  \bar{f}'(s) &=
  \begin{cases}
     0 &\text{if $s < 0$},\\
     e^s &\text{if $s > 0$},
  \end{cases}\\
  \label{eq:fbpp}
  \bar{f}''(s) &=
  \begin{cases}
     0 &\text{if $s < 0$},\\
     \delta_0 &\text{if $s = 0$},\\
     e^s &\text{if $s > 0$},
  \end{cases}
\end{align}
where $\delta_0$ denotes Dirac's delta measure at $0$
and we consider the distributional derivative for $\bar{f}''$.
Note that $f' = \lambda * \bar{f}'$ and $f'' = \lambda * \bar{f}''$.
It holds that $f \ge 1$ and $f', f'' \ge 0$.

We pick a $C^\infty$ function, $h$, on $\R$ such that
\[
h(r) =
\begin{cases}
  1 + e^r &\text{if $r \le -1$},\\
  2e^r &\text{if $r \ge 1$},
\end{cases}
\qquad h \ge 1, \ h', h'' > 0 \ \text{on $\R$.}
\]

Let $b$ be a positive constant and $F$ the $C^\infty$ function on $\R^2$ defined by
\[
F(r,t) := b \,e^{R(r)} f(t-R(r)).
\]
(We may take $b := 1$ in this section.  $b$ is needed in the later sections.)
Note that $F_t, F_{tt} \ge 0$,
where $F_t$, $F_{tt}$ are the partial derivatives of $F$.
Note also that $F = bc\,e^t$ for all $t \ge 4$.

On $Y$, we consider the metric
\begin{equation}\label{eq:metric}
g = dr^2 + h(r)^2 \biggl( dt^2 + b^2\,e^{2R(r)} \, d\rho^2 +  F(r,t)^2 \, d\tau^2 \biggr),
\end{equation}
where $d\rho^2 = \sum_{\alpha=1}^l d\rho_\alpha^2$ is the $l$-dimensional Euclidean metric
and $d\tau^2 = \sum_{\beta=1}^m d\tau_\beta^2$ the $m$-dimensional Euclidean metric.
Let us set
\[
g = \sum_{i=1}^n g_i \, dx_i^2,
\]
i.e., $x_1 := r$, $x_2 := t$, $x_{2+\alpha} := \rho_\alpha$, and $x_{2+l+\beta} := \tau_\beta$,
$g_1 := 1$, $g_2 := h(r)^2$, $g_{2+\alpha} := H(r)^2$, $H(r) := b\,e^{R(r)} h(r)$,
$g_{2+l+\beta} := h(r)^2 F(r,t)^2$ for $\alpha=1,2,\dots,l$ and $\beta=1,2,\dots,m$.

Note that, for $t \ge 4$, 
\[
dt^2 + F(r,t)^2 d\tau^2 = dt^2 + bc\,e^{2t} d\tau^2
\]
is a hyperbolic metric.

We calculate the Christoffel symbols:
\begin{align*}
\Gamma_{12}^2 &= \frac{h'}{h},
&\Gamma_{1,2+\alpha}^{2+\alpha} &= \frac{H'}{H},\\
\Gamma_{1,2+l+\beta}^{2+l+\beta} &= \frac{F_r}{F} + \frac{h'}{h},\\
\Gamma_{22}^1 &= -hh',
&\Gamma_{2,2+l+\beta}^{2+l+\beta} &= \frac{F_t}{F},\\
\Gamma_{2+\alpha,2+\alpha}^1 &= -HH',
&\Gamma_{2+l+\beta,2+l+\beta}^1 &= -h^2 F F_r - h h' F^2,\\
\Gamma_{2+l+\beta,2+l+\beta}^2 &= -F F_t
\end{align*}
for $\alpha=1,2,\dots,l$ and $\beta=1,2,\dots,m$.
We have the symmetry $\Gamma_{ij}^k = \Gamma_{ji}^k$.
Except for this, the rest of $\Gamma_{ij}^k$ are zero.

We then calculate the curvature tensor
\[
{R_{ijk}}^l = \partial_i\Gamma^l_{jk} - \partial_j\Gamma^l_{ik}
+ \sum_{m=1}^n (\Gamma^l_{im}\Gamma^m_{jk} - \Gamma^l_{jm}\Gamma^m_{ik}),
\qquad R_{ijkl} = R_{ijk}^l g_l
\]
in the following:
\begin{align*}
R_{1221} = R_{1,2+\alpha,2+\alpha,1} &= - h h'' \\
R_{1,2+l+\beta,2+l+\beta,1} &= - h h'' F^2 - h^2 F F_{rr} - 2 h h' F F_r \\
R_{1,2+l+\beta,2+l+\beta,2} &= - h^2 F F_{rt} \\
R_{2,2+\alpha,2+\alpha,2} &= - h h' H H' \\
R_{2,2+l+\beta,2+l+\beta,2} &= - h^3 h' F F_r - h^2 (h')^2 F^2 - h^2 F F_{tt} \\
R_{2+\alpha,2+\alpha',2+\alpha',2+\alpha} &= - H^2 (H')^2 \\
R_{2+\alpha,2+l+\beta,2+l+\beta,2+\alpha} &=
- h^2 F F_r H H' - h h' F^2 H H' \\
R_{2+l+\beta,2+l+\beta',2+l+\beta',2+l+\beta} &=
- h^4 F^2 F_r^2 - 2 h^3 h' F^3 F_r - h^2 (h')^2 F^4 \\
&\quad - h^2 F^2 F_t^2
\end{align*}
for $\alpha, \alpha' = 1,2,\dots,l$ with $\alpha < \alpha'$
and $\beta,\beta' = 1,2,\dots,m$ with $\beta < \beta'$.
We have the (skew-)symmetry for $R_{ijkl}$.
Except for this, the rest of $R_{ijkl}$ are zero. 
Note that the nonzero $R_{ijkl}$ are only of the form $R_{ijji}$ and
$R_{1jj2}$ up to the (skew-)symmetry.

The sectional curvatures for the plane spanned by
$\{\frac{\partial}{\partial x_i}, \frac{\partial}{\partial x_j}\}$ are:
\begin{align*}
K_{12} &= -\frac{h''}{h},\\
K_{1,2+\alpha} &= -\frac{H''}{H},\\
K_{1,2+l+\beta} &= -\frac{F_{rr}}{F} - \frac{2 h' F_r}{hF} - \frac{h''}{h},\\
K_{2,2+\alpha} &= -\frac{h' H'}{h H},\\
K_{2,2+l+\beta} &= -\frac{F_{tt}}{h^2 F} - \frac{h' F_r}{hF} - \frac{(h')^2}{h^2},\\
K_{2+\alpha,2+\alpha'} &= -\frac{(H')^2}{H^2},\\
K_{2+\alpha,2+l+\beta} &= -\frac{F_r H'}{F H} - \frac{h' H'}{h H},\\
K_{2+l+\beta,2+l+\beta'} &=
- \frac{F_r^2}{F^2}
- \frac{2 h' F_r}{hF}
- \frac{F_t^2}{h^2 F^2}
- \frac{(h')^2}{h^2}
\end{align*}
for $\alpha,\alpha'=1,2,\dots,l$ and $\beta,\beta'=1,2,\dots,m$
with $\alpha < \alpha'$ and $\beta < \beta'$.

To estimate the sectional curvatures, we prepare a few lemmas. 
\begin{lem} \label{lem:f}
\begin{enumerate}
\item $f' \le f$.
\item $f' \le f''$.
\end{enumerate}
\end{lem}

\begin{proof}
We see $\bar{f'} \le \bar{f}$ and $\bar{f}' \le \bar{f}''$
from \eqref{eq:fbp} and \eqref{eq:fbpp}.
Taking the convolution product of them with $\lambda$ yields the lemma.
%
\end{proof}

\begin{lem} \label{lem:key}
  \begin{enumerate}
  \item $h \ge 1$, $h', h'' > 0$.
  \item $H, H', H'' > 0$.
  \item $F > 0$, $F_t, F_{tt}, F_r \ge 0$.
  \item The following functions are all uniformly bounded:
  \[
    \frac{f'}{f},\ \frac{f''}{f},\ \frac{h'}{h}, \ \frac{h''}{h}, \ \frac{H'}{H}, \ \frac{H''}{H},
    \ \frac{F_r}{F}, \ \frac{F_{rr}}{F}, \ \frac{F_t}{F}, \ \frac{F_{tt}}{F}.
  \]
  \end{enumerate}
\end{lem}

\begin{proof}
(1) is obvious.

We prove (2).
The derivative of $H = b h\,e^R$ is
\[
H' = bh e^R R' + bh' e^R,
\]
which is positive.
Differentiating it again, we see
\[
H'' = be^R(hR'' + h(R')^2 + 2h'R' + h'') \ge be^R(h R'' + h'').
\]
If $r \le 1$, then $h R'' + h'' = h'' > 0$.
If $r > 1$, then $h R'' + h'' > -h/2 + h'' = e^r > 0$.
Therefore $H''$ is positive everywhere.

We prove (3).  It is clear that $F > 0$.
It follows from $f', f'' \ge 0$ that $F_t, F_{tt} \ge 0$.
We see
\[
F_r = b\, R' e^R ( f(t-R) - f'(t-R)),
\]
which is nonnegative by Lemma \ref{lem:f} and $R' \ge 0$.

(4) is clear.

We prove (5).
The boundedness of $f'/f$ and $f''/f$ follow from \eqref{eq:f}.
The boundedness of $h'/h$, $h''/h$, $T'/T$, and $T''/T$ are derived from their definitions.
We see
\begin{align*}
  \frac{H'}{H} &= \frac{hR' + h'}{h},
  &\frac{H''}{H} &= \frac{hR'' + h(R')^2 + 2h'R' + h''}{h},\\
  \frac{F_r}{F} &= \frac{R'(f-f')}{f},
  &\frac{F_{rr}}{F} &= \frac{(R')^2(f-2f' + f'') + R''(f-f')}{f},\\
  \frac{F_t}{F} &= \frac{f'}{f},
  &\frac{F_{tt}}{F} &= \frac{f''}{f},
\end{align*}
which are all bounded.
This completes the proof of the lemma.
\end{proof}

\begin{lem}\label{lem:K_ij}
There is a constant $C<0$ such that 
$C \le K_{ij} <0$ for all $i \neq j$.
\end{lem}

\begin{proof}
The lemma is readily seen from Lemma \ref{lem:key}
except the negativity of $K_{1,2+l+\beta}$.
We remark that $F_{rr} \ge 0$ does not hold.
We have
\[
K_{1,2+l+\beta} = -\frac{\varphi}{fh},
\]
where
\[
\varphi := h (R')^2 (f - 2f' + f'') + h R'' (f - f') + 2 R' h' (f - f') + h'' f.
\]
If $r \le 1$, then $R = r$, which together with Lemma \ref{lem:f}
implies $\varphi > 0$.
If $r \ge 1$, then $h = 2e^r$ and so
\[
\frac{\varphi}{h} = (R')^2 (f - 2f' + f'') + R'' (f - f') + 2 R' (f - f') + f
\]
By $R' \ge 0$, $-1/2 \le R'' \le 0$, and by Lemma \ref{lem:f}, we obtain
\[
\frac{\varphi}{h} \ge R'' (f-f') + f \ge R'' f + f \ge \frac{f}{2} > 0.
\]
Therefore, $K_{1,2+l+\beta}$ is negative.
\end{proof}

Let $\sigma$ be any $2$-plane (i.e., two-dimensional linear subspace) in
the tangent space at any point of $Y$,
and take an orthogonal basis, $\{u,v\}$, of $\sigma$.
Since $\|u \times v\| = \|u\| \cdot \|v\|$,
the sectional curvature for $\sigma$ is
\[
K_\sigma = \frac{\langle R(u,v)v,u \rangle}{\|u \times v\|^2}
= \frac{\sum_{i,j,k,l} u^i v^j v^k u^l R_{ijkl}}{\sum_{a,b} (u^a)^2(v^b)^2 g_a g_b},
\]
where $u = \sum_i u^i\frac{\partial}{\partial x_i}$,
$v = \sum_j v^j \frac{\partial}{\partial x_j}$.

\begin{lem}\label{curvature}
There is a constant $C<0$ such that 
$C \le K_\sigma <0$ for all $\sigma$.
\end{lem}

\begin{proof}
As we pointed out, the nonzero $R_{ijkl}$ are only of the form $R_{ijji}$ and
$R_{1jj2}$ up to the (skew-)symmetry for our metric $g$.
Also, $R_{ijji} < 0$ by Lemma \ref{lem:K_ij}.
Therefore, for the proof of the negativity of $K_{\sigma}$,
it suffices to prove that
\[
(u^1)^2 (v^j)^2 R_{1jj1}
+ 2 u^1 (v^j)^2 u^2 R_{1jj2}
+ (u^2)^2 (v^j)^2 R_{2jj2} < 0
\]
for all $u^1,u^2,v^j \neq 0$,
where $j = 2+l+\beta$.
This is equivalent to
\begin{equation} \label{eq:R}
(R_{1jj2})^2 < R_{1jj1} R_{2jj2}.
\end{equation}
We see $(R_{1jj2})^2 = h^4 F^2 F_{rt}^2$ and
\begin{align*}
&R_{1jj1} R_{2jj2}\\
 &=
(hh'' F^2 + h^2 F F_{rr} + 2 hh'FF_r)
(h^3 h' F F_r + h^2 (h')^2 F^2 + h^2 F F_{tt})
\end{align*}

We first assume $r \le 1$.  Note that $R = r$ in this case.
For \eqref{eq:R}, it is sufficient to prove $F_{rt}^2 \le F_{rr}F_{tt}$.
Since $F_{rt} = b e^R(f'-f'')$,
$F_{rr} = b e^R(f - 2f' + f'')$,
and $F_{tt} = b e^R f''$, the inequality $F_{rt}^2 \le F_{rr}F_{tt}$
follows from Lemma \ref{lem:f} and $f \ge 1$.

We next assume $r > 1$.  In this case, we see $h = 2e^r$,
so that \eqref{eq:R} boils down to
\begin{align} \label{eq:RR}
&(R')^2 (f''-f')^2 \\
&< [f + (R''+(R')^2)(f-f') + (R')^2(f''-f') + 2R'(f-f')] \notag\\
&\quad\times [4e^{2r}(f-f') + 4 e^{2r}f + f'']. \notag
\end{align}
We have $f + (R''+(R')^2)(f-f') > 0$ by $R'' \ge -1/2$.
We also have $(R')^2(f''-f')^2 \le (R')^2(f''-f') \times f''$.
Therefore, \eqref{eq:RR} is obtained.
The negativity of $K_{\sigma}$ has been proved.

We prove the boundedness of $K_\sigma$.
It suffices to prove the boundedness of
each
\[
A_{ijkl} := \frac{|u^i v^j v^k u^l R_{ijkl}|}{\sum_{a,b} (u^a)^2(v^b)^2 g_a g_b}.
\]
We have, for all $i < j$,
\[
A_{ijji} \le \frac{|R_{ijji}|}{g_i g_j} = |K_{ij}|,
\]
which is bounded by Lemma \ref{lem:K_ij}.
Let $j := 2+l+\beta$.
If $u^1 u^2 = 0$, then $A_{1jj2} = 0$.
For $u^1 u^2 \neq 0$, setting $s := |u^1/u^2|$, we have
\begin{align*}
A_{1jj2} &\le \frac{|u^1 u^2 R_{1jj2}|}{\sum_a (u^a)^2 g_a g_j}
\le \frac{|R_{1jj2}|}{(s g_1 + (1/s) g_2) g_j}\\
&= \frac{h^2 R' (f'-f'')}{(s+h^2/s)f}
\le \frac{h R' (f'-f'')}{2f},
\end{align*}
which is bounded since $f'-f''$ has compact support.
This completes the proof of Lemma \ref{curvature}.
\end{proof}

\subsection{Properties of $g$}
Let $b,c$ be the constants that previously appeared.

\begin{prop}\label{estimate}
For $l, m >0$, there exists a Riemannian metric $g$ on 
$Y=\Bbb R \times \Bbb R \times \R^l \times \R^m$
that is invariant by $\isom(\R^l)$ and $\isom(\R^m)$ 
satisfying the following {\rm(1)--(7)}.
\begin{enumerate}
\item There is a constant $C <0$ such that the sectional curvature $K$
satisfies $C \le K <0$ on $Y$.
\item Let $T^l, T^m$ be flat tori obtained as quotients of
$\R^l, \R^m$ by isometries. Then $g$ defines a metric on $\R \times \R \times T^l
\times T^m$ such that the volume of the following subset is finite:
 $$\{\,(r,t,\rho,\tau) \mid r \in (-\infty,-1],\ t \in [r-1,2],\ \rho \in T^l,
\ \tau \in T^m\,\}.$$
\item For $r \le 0$ and  $t \le r-1$, 
$$g= dr^2 + h(r)^2 (dt^2 + b^2e^{2r} d\rho^2 + b^2 e^{2r} d\tau^2).$$
\item For  $r\ge0$ and $t \le -1$, 
$$g= dr^2 + h(r)^2 (dt^2 + b^2e^{2R(r)} d\rho^2 + b^2e^{2R(r)} d\tau^2).$$
\item For  $r \in \Bbb R$ and $t \ge 4$,
$$g=dr^2 + h(r)^2(dt^2 + b^2e^{2R(r)}d\rho^2 +b^2c^2e^{2t} d\tau^2).$$
\item 
For $r \ge 5$, $g$ is a warped metric of the form: 
$$g=dr^2+4e^{2r} \hat{g},$$
where $\hat{g}$ is the metric on $\R \times \R^l \times \R^m$ defined by
\[
\hat{g} := dt^2 + b^2 e^6 d\rho^2 + b^2e^6 f(t-3)^2 d\tau^2.
\]

\item The metric $\hat{g}$ in {\rm(6)} has non-positive curvature.
\end{enumerate}

\end{prop}

\begin{rem}
\begin{enumerate}
\item[(i)]
$C$ does not depend on $l, m$.
\item[(ii)]
By (5), 
for all $r$ and for $ t \ge 4$ the metric is
$$ g= dr^2 + h(r)^2 (d_{hyp} +b^2 e^{2R(r)} d\rho^2),$$
where $d_{hyp} := dt^2 + b^2e^{2t} d\tau^2$ is a hyperbolic metric
with $K=-1$.
\item[(iii)]
In the proof of Theorem \ref{metric}, setting $l=m=1$, 
$g$ will be used to 
put a Riemannian metric  on a neighborhood of a boundary 
component of $\Bbb R \times P$, where $P=\Sigma \times S^1$ is a piece of the flip-manifold 
$M$.  Outside of the neighborhood, we 
use a metric from a hyperbolic metric on $\Sigma$, which coincides
with the metric $g$ at $t=4$ as in (ii).
For Theorem \ref{high.dimension}, the general form of $g$ is used.

\end{enumerate}
\end{rem}

\begin{proof}
Let $g$ be the metric given by  (\ref{eq:metric}).
\\
(1) By Lemma \ref{curvature}.
\\
(2) Without loss of generality we may assume that 
$\vol(T^l)=1, \vol(T^m)=1$ with respect to $d\rho^2, d\tau^2$,
respectively, since the volume of the concerned set
is proportional to the product $\vol(T^l) \vol(T^m)$ because of the form of $g$. 

For $r \le -1$, we have $h(r)=1+ e^{r}$ and $R(r)=r$. 
We divide the subset into two according to $t$:
\\
(i) The part for $t \in [r+1,2]$. 
Since $t-R(r)=t-r \ge 1$, we have  $f(t-R(r))=ce^{t-R(r)}$, hence 
$$g = dr^2+ (1+e^{r})^2(dt^2 + b^2e^{2r} d\rho^2 + b^2c^2e^{2t} d \tau^2).$$
Fix $r$. 
The metric $dt^2 + b^2c^2e^{2t} d \tau^2$ is hyperbolic, 
and its volume
for the part $t \in [r+1, 2]$, $\tau \in T^m$
is at most $b^m c^m \int _{-\infty}^2 e^{mt}\, dt = b^m c^m e^{2m}$.
Hence volume of the part $t \in [r+1, 2], \rho \in T^l, \tau \in T^m$
for the metric $dt^2 + b^2e^{2r} d\rho^2 + b^2c^2e^{2t} d \tau^2$
is at most  $b^{l+m} c^m e^{2m} e^{lr}$. 
Now the $g$-volume for the part 
$r \le -1, t \in [r+1, 2], \rho \in T^l, \tau \in T^m$
is, since $1+e^{r} \le 2$,  at most 
$2^{l+m+1} b^{l+m} c^m e^{2m} \int_{-\infty}^{-1} e^{lr} \, dr
= 2^{l+m+1} b^{l+m} c^m e^{2m-l}/l$.
\\
(ii) The part for $t \in [r-1, r+1]$.
In this part, we have 
$t-R(r) = t - r \in [-1, 1]$, so $f(t-R(r)) \le ce$.
The metric is 
$$g= dr^2+ (1+e^{r})^2 (dt^2 + b^2e^{2r} d\rho^2 + b^2e^{2r} f(t-R(r))^2 d\tau^2).$$
Since the volume of $be^r f(t-R(r)) T^m$ is at most $b^m e^{mr} c^m e^m$, 
the volume for $(t,\tau), t \in [r-1,r+1], \tau \in T^m$ is at most
$2 b^m c^m e^m e^{mr}$,
so that   the volume of 
$dt^2 + b^2e^{2r} d\rho^2 +b^2 e^{2r} f(t-R(r))^2 d\tau^2$
is at most $2b^{l+m} c^m e^m e^{(l+m)r}$. 
Finally, the volume of this part is, since $1+e^r \le 2$, at most 
\[
2^{l+m+2} b^{l+m} c^m e^m \int_{-\infty}^{-1} e^{(l+m)r}\,dr
= \frac{2^{l+m+2} b^{l+m} c^m e^{-l}}{l+m}.
\]

Combining (i) and (ii), volume  of the subset is at most\\
\[
2^{l+m+1} b^{l+m} c^m e^{-l} \left(\frac{e^{2m}}{l} + \frac{2}{l+m}\right).
\]

(3) 
We fix $r \le 0$. Then $R(r)=r$.
For $t \le r-1$, we have $t- R(r)= t-r \le -1$, 
so that $f(t-R(r))=1$. 
Thus, 
$g=dr^2+ h(r)^2(dt^2 + b^2e^{2r} d\rho^2 + b^2e^{2r} d\tau^2)$.

(4)
Fix $r\ge0$. Note that then $0 \le R(r) \le 3$.
So, if $t \le -1$ then $t-R(r) \le -1$, so that 
$f(t-R(r)) =1$. Substitute them to the definition of $g$.

(5) $R(r) \le 3$. Since $t \ge 4$, we have $t-R(r) \ge 1$, so that 
$f(t-R(r)) = ce^{t-R(r)}$. Substitute this to the definition of $g$.

(6)  If $r \ge 5$, then $R(r)=3$, $h(r)= 2e^{r}$ and $f(t-R(r))=f(t-3)$, so that 
$g=dr^2 + 4e^{2r} (dt^2 + b^2 e^6 d\rho^2 + b^2e^6 f(t-3)^2 d\tau^2)$,
which is a desired warped metric.

(7) follows from $f'' \ge 0$.
This completes the proof.
\end{proof}

\subsection{Proof of Theorem \ref{metric} where the pieces are products}
\proof

By assumption the graph manifold $M$ has a geometrization, ie, 
each piece has a locally product Riemannian metric of type $\Bbb H^2 \times \R$, 
and the gluing maps are isometries.
In the following, we first give an argument assuming that
$M$ has a geometrization with respect to a product metric on each piece. 
Then we will explain that in fact our argument applies
to the locally product case as well.

{\it Step 1}.
Let $P_i$ be the pieces of $M$.
Suppose $P_i = \S_i \times S^1$.
We will put a Riemannian metric on each $\Bbb R \times P_i$ so that 
they match up for gluing along boundary, which defines a Riemannian metric on $\Bbb R \times M$.
First, put a complete, hyperbolic metric of finite volume in the interior
of each $\S_i$. Let $\vol_{hyp}(\S_i)$ denote its volume. 
There is a constant $L >0$,  such that the interior of 
each $\S_i$ contains a compact subset $K_i$ homeomorphic to $\S_i$ such that 
each connected component of $\S_i \backslash K_i$ is 
isometric to an annulus $ (- \infty,0) \times S^1(bce^{2}L)$
with the metric 
$dt^2 + e^{2t} d\tau^2$,
i.e., the warped product $ (- \infty,0) \times_{e^t} S^1(bce^{2}L)$,
where $S^1(a) := \R/a\Z$ is a circle of length $a > 0$.

{\it Step 2}.
For each $r \in \Bbb R$, we consider a Riemannian product
$$K_i \times S^1(be^{R(r)-2}L)$$ then further take a 
``generalized'' warped 
product with $\Bbb R$ 
 as follows:
$$J_i = \Bbb R \times_{h(r)}  (K_i \times S^1(be^{R(r)-2}L)),$$
where at each $r$, the metric of the fiber $K_i \times S^1(be^{R(r)-2}L)$
is rescaled by $h(r)$.
We say this is a generalized warped product since 
the metric on the fiber at $r$ depends on $r$.
Then

\begin{lem}\label{core.estimate}
\begin{enumerate}
\item The subset of $J_i$ for the part $r<0$
has finite volume, which is  bounded above
by
$8beL\vol_{hyp}(\Sigma_i)  $.
\item
For the part $r>5$,  $J_i$ is a warped product:
$$(\,5,+\infty\,) \times_{2e^{r}} (K_i \times S^1(beL))$$
\item
The sectional curvature of $J_i$ is bounded:
$$ C \le K <0, $$
where $C<0$ is the constant from Proposition \ref{estimate}.
\item
Each boundary component of $J_i$  is isometric
to 
$$\Bbb R \times_{h(r)} (S^1(bce^{2}L) \times S^1(be^{R(r)-2}L)).$$
\end{enumerate}
\end{lem}
\proof
(1) At each $r<0$, $R(r)=r$, hence the volume of $K_i \times S^1(be^{R(r)-2}L)$
is $\le \vol_{hyp}(\S_i) \cdot be^{r-2} L $. 
Since $h(r) \le h(1) \le 2e$ for $r\le 0$, 
the volume of $J_i$ for the part $r \le 0$
is 
$$\le (2e)^{3} \vol_{hyp}(\Sigma_i)   L b\int _{-\infty}^0 e^{r-2} dr
= 8ebL \vol_{hyp}(\Sigma_i) .$$

(2) Suppose $r>5$.
 Then $R(r)=3, h(r)=2e^{r}$. Substitute them to the definition 
 of the metric on $J_i$.

(3) The metric of $J_i$ is written as
$$g=dr^2 +h(r)^2 (d_{hyp} + e^{2R(r)} d\rho^2),$$
where $\rho$ is for $S^1(bLe^{-2})$.
Now this metric and the metric that appears
in Proposition \ref{estimate} (5)
are locally isometric to each other (see (ii) of the remark there), but that metric
satisfies $C \le K <0$ for the constant $C$
in the proposition. 

(4). This is because each boundary of $K_i$ is isometric to $S^1(bce^{2}L)$.
\qed

{\it Step 3}.
We set $l=m=1$ in Proposition \ref{estimate}.
We prepare a manifold with boundary 
$$A=\{(r,t,\rho,\tau)| r \in \Bbb R, t \in \left[R(r)-2,4 \right], \rho \in S^1(Le^{-2}),
\tau \in S^1(Le^{-2})\}$$ with the metric $g$ given in (\ref{eq:metric}):
\[
g =dr^2 + h(r)^2 ( dt^2 + b^2e^{2R(r)} \, d\rho^2 + b^2e^{2R(r)} f(t-R(r))^2 \, d\tau^2 ).
\]

The manifold $A$ has two boundary components,
$\partial_0 A, \partial_1 A$, where
$\partial_1 A$ is the component at $t=4$ and $\partial _0 A$ at $t=R(r)-2$.
For $t=4$, we have $f(t-R(r))=f(4-R(r))=ce^{4-R(r)}$, 
so that 
$\partial _1 A$ is isometric to
\begin{equation}\label{boundary1A}
\Bbb R \times_{h(r)}  (S^1(be^{R(r)-2}L) \times S^1(bce^{2}L)).
\end{equation}
Hence  $\partial_1A$ is isometric to each boundary component of every $J_i$
by Lemma \ref{core.estimate} (4), 
so that 
we are able to glue $A$ to the boundary component of $J_i$ 
along $\partial _1 A$.
By Proposition \ref{estimate} (5) (see also the remark (ii) after that),  no singularity of the metric occurs by this gluing. 
In this way we obtain a Riemannian manifold diffeomorphic to $P_i$ (or a Riemmanian metric on $P_i$),
such that 
\begin{itemize}
\item
$P_i $ is diffeomorphic to $\Bbb R \times (\S_i \times S^1)$, where
the first parameter is $r$. 
\item
every connected component of the boundary  of $P_i$
is isometric to 
\begin{equation}\label{boundary0A}
\partial _0 A = \Bbb R \times _{h(r)}  (S^1(be^{R(r)-2} L) \times S^1(be^{R(r)-2} L)),
\end{equation}
and moreover the $1$-neighborhood of $\partial _0 A$ is 
isometric to the direct Riemannian 
product $\partial _0 A \times [0,1]$
since $f(t-R(r))=1$  for $t \in [R(r)-2, R(r)-1]$.

\item
 volume of the subset $P_i$ for the part  $r \le -1 $
is finite (since by Proposition \ref{estimate} for the part isometric to $A$ and for $J_i$ it is by Lemma \ref{core.estimate} (1)).
\item
$C \le K <0$ on $P_i$ (for $A$ by Proposition \ref{estimate}, and for 
$J_i$ by Lemma \ref{core.estimate} (3))
\item
 the metric on $P_i$ is a warped product
w.r.t. the function $2e^r$ 
for $r>5$, (for $A$ by Proposition \ref{estimate} (6), for $J_i$ 
by  Lemma \ref{core.estimate} (2))
\end{itemize}

{\it Step 4}.
Now our metric on $\R \times P_i$ will give a Riemannian metric 
on $\R \times M$. Indeed, by the second bullet in the above, 
the two boundary circles have the same length at each $r$,  
so that  we can  glue the $\R \times P_i$'s
by the given gluing maps at each $r$.

We finish the proof by checking this metric satisfies all the properties
in Theorem \ref{thm.ontaneda}. By the third bullet, 
volume of the part $(-\infty, -1] \times M$  is finite since there are only 
finitely many pieces for $M$, which implies that the volume 
for $(-\infty, 0] \times M$ is finite since $M$ is compact. 
The sectional curvature $K$ satisfies $C \le K<0$
on $\Bbb R \times M$ by the fourth bullet. The metric is a warped
product for $r\ge 5$ w.r.t. the function $2e^{r}$ and some 
metric $g_M$ on $M$ by the last bullet
and Proposition \ref{estimate} (6).
Now we  rescale the metric $g_M$ to $(1/4)g_M$, 
which 
we still denote by $g_M$,  then 
the warping function becomes $e^r$. Then we 
have $g=dr^2 + e^{2r} g_M$ for  $r\ge 5$. Set $D=5$.
Finally since $\dim M=3$, we are done. 

The proof of Theorem \ref{metric} is complete
in the case without Seifert fibered spaces, provided that 
$M$
has a geometrization with respect to a product metric 
on every piece.

{\em Locally product case.}
Now, suppose some pieces are only locally Riemmanian
product. We handle this case by following the product case, and 
we only explain the changes we need to make.
Let $P_i=\Sigma_i \times S^1$ (the trivial bundle) be a piece
which is a locally Riemannian product with respect to which $M$ has a geometrization.
Let 
$$\theta_i:\pi_1(\Sigma_i) \to S^1$$
be the monodromy representation defined by the Riemannian metric on $P_i$.

No change is necessary in Step 1.
In Step 2, instead of the Riemannian product $K_i\times S^1(be^{R(r)-2}L)$, 
we take the locally Riemannian product with respect to $\theta_i$,
which we denote by
$$K_i \times_{\theta_i}  S^1(be^{R(r)-2}L).$$
Accordingly we also use $K_i \times_{\theta_i}  S^1(be^{R(r)-2}L)$
in the statement of Lemma \ref{core.estimate}, but the proof is nearly same:
for example in the proof of (3), 
$g=dr^2 +h(r)^2 (d_{hyp} + e^{2R(r)} d\rho^2)$ does not hold any more, but 
$g$ is only locally isometric to the right hand side. But this is enough
since the sectional curvature depends only locally on $g$.

In Step 3, when we define the manifold $A$, we use the same definition, but 
the metric on $A$ is a locally  product metric
with respect to the monodromy $\theta_i$ on the fiber circle for $\rho$.
We call this circle $\rho$-circle in the following.
Accordingly, in the description (\ref{boundary1A}), $\partial_1A$ becomes only a locally Riemannian product
with respect to $\theta_i$ on the $\rho$-circle (which is the first $S^1$
acted by the second $S^1$ via $\rho$).
This also happens in the metric description of $\partial _0A$
in (\ref{boundary0A}).

Finally in Step 4, the two circles in (\ref{boundary0A})
has same length in this case, and we kept using the same
monodromy $\theta_i$ on each piece $P_i$, 
therefore, the given gluing maps are all isometric.
This finishes the proof in this case, and 
the proof of Theorem \ref{metric}
for flip manifolds without Seifert space pieces is complete. 
\qed

\subsection{Proof of Theorem \ref{metric} for the general case}
We now handle a graph manifolds such that possibly some pieces are Seifert fibered spaces or fibers are non-orientable
(from now on we consider a Seifert fibered space contains
the latter case).
The argument is  identical to  the previous case
with non-trivial monodromy representation $\theta_i$
of $\pi_1(\Sigma_i)$ where $\Sigma_i$ is the base surface of 
a piece $P_i$. 
 The only difference is that $\Sigma_i$ is maybe an orbifold and 
 $\pi_1(\Sigma_i)$ is the orbifold fundamental group. 
 In the following we only explain that part. 
 A good reference for the geometry of Seifert fibered spaces
 is \cite{scott}. 

\proof 
Let $P$ be a piece in $M$. Suppose $P$ is  a Seifert fibered space,
otherwise we do not have to change anything. We remember that 
when $P$ is a trivial circle bundle over a surface,
we can choose the length of the fiber circle when we put a locally product
Riemannian metric. 

Let $\Sigma$ be the base orbifold of $P$.
Let $x_1, \cdots, x_n$ be the singular points of $\Sigma$ such that 
the twist parameter at $x_i$ is $q_i/p_i$.
Since $P$ has non-empty boundary, $P$ admits the geometry of  
$\Bbb H^2 \times \Bbb R$ (\cite[Theorem 5.3(ii)]{scott}).
We explain this part in some details (cf. \cite[Proof of Theorem 5.3(ii)]{scott}, \cite[Lemma 2.5]{L}). 
We put $\Sigma$ a complete hyperbolic orbifold metric of finite volume, then view
$P$ as an $S^1$-bundle over the orbifold $\Sigma$ with
a Riemannian metric that is 
locally isometric to $\Bbb H^2 \times \Bbb R$.
The global geometry is described by the monodromy representation 
of $\pi_1(\Sigma)$ into the group $S^1$ if the fibers
are oriented, otherwise into $S^1 \rtimes \Bbb Z_2$,
the isometry group of a circle. Here the fundamental group is in the orbifold sense, and $\Z_2$ means
$\Z/2\Z$.

First, assume $\Sigma$ is orientable. 
Let $X_i$ denote a loop around the singular point $x_i$, and 
$b_1, \cdots, b_n$ the curves around the punctures (boundary components) of $\Sigma$.
Let $g$ be the genus of $\Sigma$ then take loops
$\alpha_1, \beta_1; \cdots; \alpha_g, \beta_g$ associated to the genus such that 
$X_i, b_i, \alpha_i, \beta_i$ generate the fundamental group of $\Sigma$
satisfying a well-known relation (after choosing orientations of the loops
suitably) :
$$\prod_i[\alpha_i, \beta_i] \prod_i X_i \prod_i b_i = 1.$$

Let $\theta(\alpha_i), \theta(\beta_i), \theta(X_i), \theta(b_i)$ denote the monodromy along those loops 
for the $S^1$-fiber.
We set for each $i$
$$\theta(X_i)=2\pi q_i/p_i.$$ 
We choose $\theta(b_i),\theta(\alpha_i), \theta(\beta_i)$ for each $i$ such that in $S^1 \rtimes \Z_2$, 
$$
\theta(\prod_i[\alpha_i, \beta_i] \prod_i X_i \prod_i b_i)=1.
$$
Then  there is a locally product Riemannian 
metric on $P$ whose monodromy representation is $\theta$.
Note that if $ \theta(\alpha_i), \theta(\beta_i) \in S^1$ 
then since $S^1$ is abelian we always have
$\theta(\prod_i[\alpha_i, \beta_i])=0$
in $S^1$.

Conversely, the monodromy representation 
induced by a locally product Riemannian metric
is obtained in the above way. 

If $\Sigma$ is not oriented, the relation in the fundamental 
group is slightly different, but the rest is same and we omit
repeating it. 

Note that when we put a Riemannian metric on $P$, as before we 
can choose the length of the $S^1$-fiber (at a regular point) as we want. 
Also each boundary component of $P$ is a flat torus/Klein bottle. 

We take a compact subset $K$ homeomorphic to $\Sigma$
such that all singular points are contained in $K$, and that 
each connected component of $\Sigma \backslash K$
satisfies the same metric property as the non-generalized case
described in Step 2 in the previous section. 
We do not need to alter the argument since we modify the 
metric only outside of $K$, then that $\Sigma$ is an orbifold
does not cause any difference.

Now we proceed in the same way as the previous case, and complete the proof of Theorem \ref{metric} in general. 
\qed

We give an example of a flip manifold with a geometrization made from a Seifert fibered space.
\begin{ex}[Seifert fibered space as a piece]\label{flat.seifert}
We give an example of a Seifert fibered space  that can appear as a piece in 
a flip manifold with a geometrization. 
Let $\Gamma$ be a three-punctured sphere. There is an obvious
action of $G=\Bbb Z / 3\Bbb Z$ rotating the three punctures with
a generator $\rho$.
Put a complete hyperbolic metric on $\Gamma$ which is 
$\rho$-invariant.
Now set $\Sigma=\Gamma/G$, which is a hyperbolic orbifold
with two singular points, $p_1, p_2$, and 
with one puncture. 
Take the product $\Gamma \times S^1$ and let $G$ act on it
such that $\rho$ acts on $S^1$ by the rotation of $2\pi/3$.
This is a free action and the quotient $(\Gamma \times S^1)/G$
is a three dimensional manifold $P$, which is a Seifert fibered space
over $\Sigma$ such that the twists at $p_1, p_2$
are $1/3, 2/3$, respectively. 
(One can say that the twist at $p_2$ is $-1/3$).
$P$ has only one boundary component, which is a Riemannain
product of the fiber circle and  a loop around the puncture of $\Sigma$
since the monodromy is trivial.
Now, for example,  we prepare another copy of this, then glue the two 
along their boundary by a trivial or flip map, and obtain a 
flip manifold which admits a geometrization. 
\end{ex}

We also record an example of a graph manifold $M$ with a geometrization whose gluing map
is not a trivial nor a flip map, cf, \cite[Example 1.5]{BSc}.
\begin{ex}[Graph manifold of non-positive curvature]\label{example.non-flip}
Consider the parallelogram of side length 1 with the angles of the corners
equal to $\pi/3, 2\pi/3, \pi/3, 2\pi/3$. Choose a vertex of angle $\pi/3$
and call it $O$, then call the adjacent vertices $A, B$. The last vertex
is called $D$. We obtain a flat torus $T$ gluing the sides $OA$ and $BD$; and $OB$ and $AD$. 
We regard $T$ as a circle bundle
over a circle where the base circle is $OB$ and the fiber circle is $OA$.
The monodromy with respect to the flat metric is $\pi$.

 $T$ has 
an interesting isometry $\phi$ that is defined by mapping:
$$OA \mapsto BA, OB \mapsto OA.$$
Notice  $\phi$ is not homotopic to the trivial map nor
the flip map of $T$.

We define a graph manifold using $\phi$. 
Let $\Sigma$ be a compact orientable surface of genus one with 
two boundary components, $a+, a-$. Orient those two curves
using the orientation of $\Sigma$. Let $P$ be a trivial circle bundle over $\Sigma$
and we put a locally product metric of type $\Bbb H^2 \times \R$
on $P$ such that the monodromy satisfies $\theta(a+)=\pi, \theta(a-)=\pi$.
We arrange that there is a small constant $c>0$ such that
the two boundary tori $T+, T-$ of $P$ at $a+,a-$, respectively, are isometric to $T$
with the metric rescaled by $c$.
Now we glue $T+$ to $T-$ by $\phi$, which is an isometry. $\phi$ is 
not a flip nor trivial map. 
In this way we obtain an oriented  graph manifold $M$ that has a Riemannian metric of 
non-positive curvature. 
\end{ex}

\begin{rem}\label{rem.buyalo}
As we said the property that a graph manifold $M$ has a geometrization formulated
differently in \cite{BK}. Although we put a complete, finite volume 
hyperbolic metric on the base surface/orbifold of a piece, they put
a hyperbolic metric with a geodesic boundary (i.e., if you lift it to 
the universal cover, then it is a geodesic in $\Bbb H^2$).
In both settings we can see the piece as a circle bundle 
over the base, and it defines a monodromy representation of the fundamental 
group of the base into $S^1$,
which coincides for the two settings. 
So, if $M$ admits isometric geometrization, then its monodromy representation
can be used to put a locally product Riemannian 
metric on each piece that gives a geometrization on $M$ in our sense.

\end{rem}

\subsection{Proof of Theorem \ref{high.dimension}}
\begin{proof}
The proof of Theorem \ref{high.dimension}
is nearly identical to the proof of the version of Theorem \ref{metric}
where  each piece is a product of a surface and a circle, which 
is exactly the case where $l=m=1$ in Theorem \ref{high.dimension}.
The main body of the argument for Theorem \ref{metric}
is Proposition \ref{estimate}, which is already shown for general $l, m$.
So we do not repeat the argument, except we make one remark.
Suppose $W_1=\Sigma_1 \times T_1^m$, $W_2=\Sigma_2 \times T_2^m$
are pieces such that $S_1^l \times T_1^m$ and $S_2^l \times T_2^m$
are glued by an isometry, where $S_i^l$ is a boundary torus of $\Sigma_i$.
Also, suppose $\Sigma_i \subset X_i$.
By taking $\Sigma_i$ larger in $X_i$  if necessary, one may assume
the metric on $S_i$ is rescaled 
by any constant $0<c<1$. Also one can rescale the fibers $T_i^m$
by the same constant $c$, which leaves the gluing isometric. 

It follows from Proposition \ref{estimate}(7) that
$M$ carries a metric of non-positive curvature,
so that $\Wh(M)$ vanishes.
This completes the proof.
\end{proof}

\section{The other family}

We discuss the other examples of manifolds
that will be ends. 
\subsection{Construction}

Let $M_1, M_2$ be $n$-dimensional closed, orientable hyperbolic manifolds with totally geodesic, orientable submanifold $N_1, N_2$, respectively, of codimension two. 
Assume that $N_1$ and $N_2$ are isometric
by an isometry $b:N_1 \to N_2$.

The unit normal bundle of $N_1$ in $M_1$
 is an $S^1$-bundle, $(X_1,N_1,S^1)$, 
 with oriented fibers, 
  which we also denote by 
 $$X_1=N_1 \ltimes S_1, \, 
 \mathrm{or} \, \, X_1 =S_1 \rtimes N_1.$$
 We will use this notation for bundles in 
 this paper, which does not mean a semi-direct product of group structures. 
 
  The metric of $M_1$
induces a Riemannian metric on this bundle which 
is locally a Riemannian product of the hyperbolic metric on $N_1$
and $S^1$.
Similarly we have an $S^1$-bundle
over $N_2$, $(X_2,N_2,S^1)$,
which is locally a Riemannian product. 

For a sufficiently small constant $\epsilon>0$, the boundary of $V_1=M_1 \backslash N_\epsilon(N)$
is canonically identified with $(X_1,N_1,S^1)$.
Also, $V_2 = M_2 \backslash N_\epsilon(N)$
is identified with $(X_2,N_2,S^1)$.

Suppose $S^1$-bundles over $M_1,M_2$
 with 
Riemannian metrics which are locally product of $M_1, M_2$, respectively,
and $S^1$ are given.
We denote them by $(Y_1,M_1,S^1), (Y_2,M_2,S^1)$, and the restriction of them to $N_1, N_2$
by $(Y_1|N_1, N_1, S^1), (Y_2|N_2, N_2, S^1)$.

We assume $(X_1,N_1,S^1)$ is isometric to $(Y_2|N_2, N_2,S^1)$ by a bundle map $(f_1,b)$
where $b$ is the isometry between $N_1$ and $N_2$, and also $(Y_1|N_1,N_1,S^1)$
 is isometric to $(X_2,N_2,S^1)$
by a bundle map $(f_2,b)$ 
in the same manner. 
It then follows that the fiber prduct
$(X_1 \times Y_1|N_1, N_1, S^1\times S^1)$
is isometric to the fiber product $(X_2 \times Y_2|N_2, N_2, S^1 \times S^1)$
by the {\it flip map} 
$$\phi: (n,(s_1, s_2)) \mapsto (b(n),(f_2(s_2),f_1(s_1))),
\quad n \in N_1, (s_1,s_2) \in S^1\times S^1,$$
or the trivial map
$$\phi: (n,(s_1, s_2)) \mapsto (b(n),(f_1(s_1),f_2(s_2))),
\quad n \in N_1, (s_1,s_2) \in S^1\times S^1.$$
Note that the metric on the fiber $S^1 \times S^1$
of the two bundles is a product metric since  $X_i$ are defined over $M_i$.

The fiber products $(X_1 \times Y_1|N_1, N_1, S^1\times S^1)$ and 
$(X_2 \times Y_2|N_2, N_2, S^1 \times S^1)$
are identified with 
the the boundary of $Y_1|V_1$
and $Y_2|V_2$.

Now we define 
$$W= (Y_1|V_1, V_1, S^1) \cup_{\phi} (Y_2|V_2,V_2,S^1)$$
by identifying their boundaries 
$(X_1\times Y_1|N_1,N_1,S^1), (X_2 \times Y_2|N_2, N_2,S^1)$
using $\phi$.

For example, if $n=2$ then $N_1,N_2$ are points
and $W$ is a flip manifold. 

We recall the theorem from the introduction. 
\begin{thm}[Theorem \ref{thm3}]\label{thm2}
Assume $W$ has a geometrization (i.e., the gluing maps are isometric).
Then 
$W$ carries a metric of non-positive curvature,
so that $\Wh(W)$ vanishes. 
Also, $\R \times W $ carries 
 a complete Riemannian metric
that is an eventually warped cusp metric with $C \le K <0$ for
some constant $C<0$.
\end{thm}

\begin{rem}
As in the construction of  3-dimensional graph manifolds, as a generalization 
of the theorem, one can use
a finite collection of codimension 2
submanifolds $N_1, \cdots, N_l$, each of which appears
two times in the union of $n$-dimensional closed hyperbolic 
manifolds $M_1, \cdots, M_k$ as totally geodesic, mutually disjoint,
submanifolds. 
For a sufficiently small $\epsilon>0$
we remove the $\epsilon$-neighborhoods of $N_i$'s, then 
glue the two boundaries of $N_\epsilon(N_i)$ by
either the trivial map or the flip map.
In this way we obtain a closed manifold $W$ for which 
Theorem \ref{thm2} holds. 

Note that Theorem \ref{metric} follows from 
the generalized version of Theorem \ref{thm2}
if all of the base surfaces have genus at least two.
\end{rem}

\subsection{Gluing condition}

We discuss the condition for a flip map to be isometric
in  the case $n=3$ in some details.
The $N_i$ are simple closed geodesics, and we denote them
by $\gamma_i$. By our assumption they have same length.
Let $m_i$ be the meridean curve for $\gamma_i$ in $M_i$.
We denote $\sigma_i$ the fiber circle of $X_i$.
With respect to the Riemannian metric, we can measure 
the monodromy (i.e., rotation) along the curve $\gamma_i$ for $\sigma_i$
and $m_i$, respectively. We denote them by 
$0 \le \theta(\sigma_i), \theta(m_i) <2\pi$.

Notice that the flip map $\phi$ is an isometry
if and only if 
\begin{equation}\label{mono.1}
\theta(m_1)=\theta(\sigma_2), \qquad \theta(\sigma_1)=\theta(m_2)
\end{equation}

In general, i.e., if $\dim N \ge 1$, then 
let $\rho_N(m_1)$ be the monodromy 
representation of $\pi_1(N)$ to $S^1$, 
in terms of the meridean curve $m_1$.
Let $\rho_N(\sigma_1)$ be the monodromy 
representation in terms of $\sigma_1$.
Similarly we define 
$\rho_N(m_2), \rho_N(\sigma_2)$.
We then assume 
\begin{equation}\label{mono.1.general}
\rho_N(m_1) =  b^* \rho_N(\sigma_2), \qquad
\rho_N(\sigma_1) = b^*  \rho_N(m_2)
\end{equation}

It is an interesting question if the bundles $X_i$ satisfying 
this property exists for given $(N_i,M_i)$.
One sufficient condition is that
$H_1(N_i, \Z)$ injects into $H_1(M_i, \Z)$
for both $i=1,2$. Indeed, if so then 
first define a circle bundle
over $N_2$ using $\rho_N(m_1) =  b^* \rho_N(\sigma_2)$ (here, we use that $S^1$ is abelian), then extend it to $M_2$ (use that 
$H_1$ injects), which will be $X_2$.
Similarly we can define $X_1$.

We realize that it is enough if $X_i$ are defined over $V_i$
for our construction.
But in this case we need an additional condition since 
the metric on 
the fiber $S^1\times S^1$ is flat, but not a Riemannian product
any more. 
Hence the monodromy  $\theta_{m_i}(\sigma_i), \theta_{\sigma_i}(m_i)$
are not trivial in general, and we need
\begin{equation}\label{mono.2}
\theta_{m_1}(\sigma_1) =\theta_{\sigma_2}(m_2), \qquad
\theta_{\sigma_1}(m_1) =\theta_{m_2}(\sigma_2)
\end{equation}
It turns out that if one is satisfied then the other one follows.
We will assume this condition if we consider 
bundles that are defined only on $V_i$.

\begin{ex}
We discuss the case that $\dim M=2, \dim N=1$.
If $X$ is defined over $M$, then the boundary
of $V$ is a torus which is a Riemannian product.
But if $X$ is defined only on $M \backslash N_\epsilon(N)$, then maybe $\theta_m(\sigma)\not=0$, 
and the boundary of $V$ is a flat torus, but 
not a product. Then we need to arrange
that $\theta_m(\sigma)$ coinsides
for a pair of tori which are identified.

\end{ex}

\subsection{Outline of proof of Theorem \ref{thm2}}
The proof of Theorem \ref{thm2} is parallel to Theorem \ref{metric}.

We denote $Y_i|V_i$ by $P_i$ and call it a {\em piece}.
$N_i$ are isometric to each other by the isometry $b$, 
so we may write them as $N$.

We will put a metric
on $J_i=\R\times P_i $ so that they match up for gluing
by ${\rm id} \times \phi$, 
which gives a desired metric on $\R \times W$
to apply Theorem \ref{thm.ontaneda}.
$J_i$ has a product metric using the (non-complete) hyperbolic metric on $V_i$, but there will be 
singularity when we glue them.
So we deform the original metric near $\partial J_i$. 
A small neighborhood of $\partial J_i$ 
is diffeomorphic to $\R \times [0,\infty) \times ((S^1 \times S^1) \rtimes N)$. In view of that we will construct a complete Riemannian metric $g$ of negative curvature on
$$\R \times \R \times S^1 \times S^1 \times N,$$
which is invariant by a rotation on each $S^1$.
We arrange that there is a constant $a$ such that
for every $r \in \R$ the metric on 
$\{r\} \times [a,\infty) \times S^1 \times S^1 \times N$,
is identical to the original product metric on $P_i$
upto scaling by a constant depending on $r$
(see Proposition \ref{estimate.2} (5)).
Here, the identification of the metric is canonically done
between the fiber bundle  $(S^1 \times S^1) \rtimes N$
and $S^1 \times S^1 \times N$ since  the metric on the 
product is invariant 
by rotations on the both $S^1$-factors.

Moreover, the metric $g$ will be defined on 
$\R \times \R \times S^1 \times S^1 \times \R^{n-2}$.
The factor $\R^{n-2}$ is identified with $\tilde N$ and $g$ is 
invariant by the action by $\pi_1(N)$ which acts trivially on the other 
factors. 
In this way, 
$(\R \times S^1 \times \R^{n-2})/\pi_1(N)$
is identified with  $N_\epsilon(N) \backslash N$.
The other $S^1$ is for the fiber circle in $P_i$, and 
we can regard $g$ as a metric on $J_i=\R \times P_i$.

We show the following (cf.~Proposition \ref{estimate}).
Recall that $S^1(a)$ is a circle of length $a$.
\begin{prop}\label{estimate.2}
Let $c_1,c_2>0$ be constants. Then 
there is a Riemannian metric $g$ on 
$\Bbb R \times \Bbb R \times S^1(c_1) \times S^1(c_2) \times N$ that is invariant by rotations on each $S^1$ 
satisfying the following {\rm(1)--(7)}.
\begin{enumerate}
\item There is an absolute constant $C <0$,
which does not depend on $c_1,c_2$, such that $C \le K <0$ on 
$\Bbb R \times \Bbb R \times S^1(c_1) \times S^1(c_2) \times N$.
\item Volume of the following subset is finite:
 $$\{\,(r,t,\rho,\tau, n) \mid r \in (-\infty,-1],\ t \in [r-1,2],\ \rho \in S^1(c_1),
\ \tau \in S^1(c_2), \, n \in N\}.$$
\item For $r \le 0$ and  $t \le r-1$, 
\begin{align*}
g = dr^2 + h(r)^2 \biggl( dt^2 + b^2\,e^{2r} \, d\rho^2 +  b^2 e^{2r} \, d\tau^2 
 +  dw^2 + \sum_{j=1}^{n-3} e^{2w} dw_j^2 \biggr).
\end{align*}
\item For  $r\ge0$ and $t \le -1$, 
\[
g = dr^2 + h(r)^2 \biggl( dt^2 + b^2\,e^{2R(r)} \, d\rho^2 +  b^2\,e^{2R(r)} \, d\tau^2
+ dw^2 + \sum_{j=1}^{n-3} e^{2w} dw_j^2 \biggr).
\]
\item For  $r \in \Bbb R$ and $t \ge a$,
\begin{align*}
g &= dr^2 + h(r)^2 \biggl( dt^2 + b^2\,e^{2R(r)} \, d\rho^2 +  \sinh^2(t-5) \, d\tau^2 \\
&\qquad\qquad + \cosh^2(t-5) \Bigl( dw^2 + \sum_{j=1}^{n-3} e^{2w} dw_j^2 \Bigr)\biggr).
\end{align*}
\item 
For $r \ge 5$, the metric $g$ is a warped metric of the form: 
\[
g = dr^2 + 4e^{2r} \hat{g},
\]
where $\hat{g}$ is the metric on $\R \times S^1(c_1) \times S^1(c_2) \times N$
defined by
\[
\hat{g} := dt^2 + b^2\,e^6 \, d\rho^2 +  \tilde{F}(r,t)^2 \, d\tau^2 
+ T(t)^2 \Bigl( dw^2 + \sum_{j=1}^{n-3} e^{2w} dw_j^2 \Bigr).
\]
Here, $\tilde{F}(r,t)$ {\rm(}and hence $\hat{g}$ too{\rm)}
is independent of $r$ for $r \ge 5$.
\item\label{it:g-hat-r} The metric $\hat{g}$ in {\rm(6)}
has non-positive curvature for $r \ge 5$.
\end{enumerate}
\end{prop}

We postpone proving this proposition and prove 
Theorem \ref{thm2} using it. 
\subsection{Proof of Theorem \ref{thm2}} 

\proof
First, $W$ carries a metric 
of non-positive curvature by Proposition \ref{estimate.2}\eqref{it:g-hat-r}.
This implies that $\Wh(W)$ vanishes, \cite{FJ}.

We now show the claim for $\R \times W$.
We closely follow each step of the argument for Theorem \ref{metric}.
But there is one additional issue and we make a remark
on that. For Theorem \ref{metric},
each piece $P$ is a trivial bundle $\Sigma \times S^1$ (for the non-general case).
We glue pieces along boundaries by isometries, and 
a boundary component of $P$ is $S^1 \times S^1$, where the first
$S^1$ is a boundary component of $\Sigma$.
On the other hand, for Theorem \ref{thm2}, 
a boundary component of a piece $P$ will be an $S^1$-bundle over
an $S^1$-bundle over a hyperbolic manifold $N$:
$S^1 \rtimes (S^1 \rtimes N)$.
But notice that any metric $g$ on $S^1 \times S^1 \times N$ that 
is invariant by rotations on both circles gives a metric 
to the boundary which is locally isometric to $g$.
In view of this, when we construct a metric (see Section \ref{section.metric}), we consider only  rotationally 
invariant ones on a product space then 
descend it to a space with circle bundle structures,
so that the bundle issue is not an extra problem for us.
In the following, we may write $S^1 \rtimes (S^1 \rtimes N)$
simply as $S^1 \rtimes S^1 \rtimes N$.

{\it Step 1}. 
Fix a small constant $\epsilon>0$.
Set $\Sigma_i= M_i - N_\epsilon(N_i)$.
The boundary of $\Sigma_i$ is a circle bundle over $N_i$.
Set $P_i=(M_i - N_\epsilon(N_i)) \ltimes S^1$.

{\it Step 2}. 
We will put a metric on $P_i$ and glue them
along the boundary.
Set $K_i=M_i - N_{2\epsilon}(N_i)$.
Then $\Sigma_i - K_i$ is isometric to
$[\epsilon,2\epsilon) \times (S^1(2\pi) \rtimes N)$
with the metric 
$$dt^2 + \sinh(t) d\tau^2 + \cosh(t) g_N$$

{\it Step 3}.
For each $r \in \Bbb R$, we consider an $S^1$-bundle which 
is locally a  Riemannian product:
$$(M_i-N_\epsilon(N_i)) \ltimes S^1(be^{R(r)-2}L),$$ then further take a 
``generalized'' warped 
product with $\Bbb R$ 
 as follows:
$$J_i = \Bbb R \times_{h(r)}  \{ (M_i-N_\epsilon(N_i))\ltimes S^1(be^{R(r)-2}L)\},$$
where at each $r$, the metric of the fiber $K_i \ltimes S^1(be^{R(r)-2}L)$
is rescaled by $h(r)$.
We say this is a generalized warped product since 
the metric on the fiber at $r$ depends on $r$.

Then we have the following lemma.
The argument is similar to Lemma \ref{core.estimate}
and we skip it.
\begin{lem}\label{core.estimate.2}
\begin{enumerate}
 \item The subset of $J_i$ for the part $r<0$
has finite volume, which is  bounded above
by
$2^{n+1} e^{n-1} b L \vol_{hyp}(M_i)$.
\item
For the part $r>5$,  $J_i$ is a warped product:
$$(\,5,+\infty\,) \times_{2e^{r}} (\R \times S^1 \rtimes S^1 \rtimes N)$$
\item
The sectional curvature of $J_i$ is bounded:
$$ C \le K <0, $$
where $C<0$ is the constant from Proposition \ref{estimate.2}.
\item
Each boundary component of $J_i$  is isometric
to 
$$\Bbb R \times_{h(r)} (S^1(bce^2L) \rtimes S^1(be^{R(r)-2}L) \rtimes N).$$
\end{enumerate}
\end{lem}

{\it Step 4}.
Similar. We use Proposition \ref{estimate.2}.
We skip details. 

{\it Step 5}.
Similar. We use Proposition \ref{estimate.2}.
We skip details.

Theorem \ref{thm2} is proved. 
\qed

\subsection{Metric construction}\label{section.metric}
We are left with proving Proposition \ref{estimate.2}.
It is done by constructing $g$.
For any constant $a > 5$,
we put $\delta := \delta(a) := (a-5)/2$ and $b := b(a) := c^{-1} e^{-(a-\delta)} \sinh(a-\delta-5)$, where we recall $c = (\lambda * e^t)(0)$.
There is a $C^\infty$ function, $\tilde{F}$, on $\R^2$ such that
\begin{enumerate}
\item[(i)] for all $r$ and $t$,
\[
\tilde{F}(r,t) =
\begin{cases}
  F(r,t) = b \,e^{R(r)} f(t-R(r))  &\text{if $t \le 5$},\\
  \sinh (t-5) &\text{if $t \ge a$},
\end{cases}
\qquad 
\]
\item[(ii)] $\tilde{F}_t, \tilde{F}_{tt} \ge 0$ everywhere, 
\item[(iii)] $\tilde{F}(r,t)$ is independent of $r$ if $t \ge 4$ or $r \ge 5$.
\end{enumerate}
Let us explain why such a function $\tilde{F}$ exists.
Assume $t \ge 4$.
Since $R \le 3$, we have $t-R \ge 1$ and so
\[
b \, e^R f(t-R) = bc\,e^t
\begin{cases}
  > \sinh(t-5) &\text{if $t < a-\delta$,}\\
  = \sinh(t-5) &\text{if $t = a-\delta$,}\\
  < \sinh(t-5) &\text{if $t > a-\delta$.}
\end{cases}
\]
Therefore, there is a $C^\infty$ approximation, $\tilde{F}$,
of the continuous function
\[
\begin{cases}
  b \,e^{R(r)} f(t-R(r))  &\text{if $t \le a-\delta$},\\
  \sinh(t-5) &\text{if $t > a-\delta$},
\end{cases}
\]
satisfying the required conditions.

Take a $C^\infty$ function, $T$, on $\R$ such that
\[
T(t) = 
\begin{cases}
  1 &\text{if $t \le 4$},\\
  \cosh(t-5) &\text{if $t \ge a$},
\end{cases}
\qquad T \ge 1, \ T', T'' \ge 0.
\]

For $n \ge 2$, we consider the metric
\[
g = \sum_{i=1}^n g_i \, dx_i^2,
\]
where $x_1 := r$, $x_2 := t$, $x_3 := \rho$, and $x_4 := \tau$, $x_5 := w$,
$x_i := w_{i-5}$ for $i \ge 6$,
$g_1 := 1$, $g_2 := h(r)^2$, $g_3 := H(r)^2$, $H(r) := b\,e^{R(r)} h(r)$,
$g_4 := h(r)^2 \tilde{F}(r,t)^2$,
$g_5 := h(r)^2 T(t)^2$, $g_i := e^{2w} h(r)^2 T(t)^2$ for $i \ge 6$.
We see that
\begin{align*}
g &= dr^2 + h(r)^2 \biggl( dt^2 + b^2\,e^{2R(r)} \, d\rho^2 +  \tilde{F}(r,t)^2 \, d\tau^2 \\
&\qquad\qquad + T(t)^2 \Bigl( dw^2 + \sum_{j=1}^{n-3} e^{2w} dw_j^2 \Bigr)\biggr),
\end{align*}
where the term 
\[
T(t)^2 \Bigl( dw^2 + \sum_{j=1}^{n-3} e^{2w} dw_j^2 \Bigr)
\]
vanishes for $n = 2$.
Note that, for $t \ge a$, 
\begin{align*}
& dt^2 + \tilde{F}(r,t)^2 d\tau^2 + T(t)^2 \Bigl( dw^2 + \sum_{j=1}^{n-3} e^{2w} dw_j^2 \Bigr)\\
&= dt^2 +\sinh^2(t-5) d\tau^2 + \cosh^2(t-5) \Bigl( dw^2 + \sum_{j=1}^{n-3} e^{2w} dw_j^2 \Bigr).
\end{align*}
is a hyperbolic metric.

We calculate the Christoffel symbols:
\begin{align*}
\Gamma_{12}^2 &= \frac{h'}{h},
&\Gamma_{13}^3 &= \frac{H'}{H},
&\Gamma_{14}^4 &= \frac{\tilde{F}_r}{\tilde{F}} + \frac{h'}{h},\\
\Gamma_{1i}^i &= \frac{h'}{h} \ \text{for $i \ge 5$},\\
\Gamma_{22}^1 &= -hh',
&\Gamma_{24}^4 &= \frac{\tilde{F}_t}{\tilde{F}},
&\Gamma_{2i}^i &= \frac{T'}{T} \ \text{for $i \ge 5$},\\
\Gamma_{33}^1 &= -HH',
&\Gamma_{44}^1 &= -h^2 \tilde{F} \tilde{F}_r - h h' \tilde{F}^2,
&\Gamma_{44}^2 &= -\tilde{F} \tilde{F}_t,\\
\Gamma_{55}^1 &= -h h' T^2,
&\Gamma_{55}^2 &= -T T',
&\Gamma_{5i}^i &= 1 \ \text{for $i \ge 6$},\\
\Gamma_{ii}^1 &= - e^{2w} h h' T^2,
&\Gamma_{ii}^2 &= -e^{2w} T T',
&\Gamma_{ii}^5 &= -e^{2w} \ \text{for $i \ge 6$}.
\end{align*}
The curvature tensor is calculated as follows, for $j > i \ge 6$:
\begin{align*}
R_{1221} &= - h h'', \\
R_{1331} &= - H H'', \\
R_{1441} &= - h h'' \tF^2 - h^2 \tF \tF_{rr} - 2 h h' \tF \tF_r, \\
R_{1442} &= - h^2 \tF \tF_{rt}, \\
R_{1551} &= - h h'' T^2, \\
R_{1ii1} &= - e^{2w} h h'' T^2, \\
R_{2332} &= - h h' H H', \\
R_{2442} &= - h^3 h' \tF \tF_r - h^2 (h')^2 \tF^2 - h^2 \tF \tF_{tt}, \\
R_{2552} &= - h^2 (h')^2 T^2 - h^2 T T'', \\
R_{2ii2} &= - e^{2w} h^2 (h')^2 T^2 - e^{2w} h^2 T T'', \\
R_{3443} &= - h^2 \tF \tF_r H H' - h h' \tF^2 H H', \\
R_{3553} &= - h h' H H' T^2, \\
R_{3ii3} &= - e^{2w} h h' H H' T^2, \\
R_{4554} &= - h^3 h' \tF \tF_r T^2 - h^2 (h')^2 \tF^2 T^2 - h^2 \tF \tF_t T T', \\
R_{4ii4} &= - e^{2w} h^3 h' \tF \tF_r T^2 - e^{2w} h^2 (h')^2 \tF^2 T^2 - e^{2w} h^2 \tF \tF_t T T', \\
R_{5ii5} &= - e^{2w} h^2 (h')^2 T^4 - e^{2w} h^2 T^2 (1+(T')^2), \\
R_{ijji} &= - e^{4w} h^2 (h')^2 T^4 - e^{4w} h^2 T^2 (1+(T')^2).
\end{align*}
The sectional curvatures are:
\begin{align*}
K_{12} &= -\frac{h''}{h},
& K_{13} &= -\frac{H''}{H},\\
K_{14} &= -\frac{\tilde{F}_{rr}}{\tilde{F}} - \frac{2 h' \tilde{F}_r}{h\tilde{F}} - \frac{h''}{h},
& K_{1j} &= -\frac{h''}{h} \ \text{for $j \ge 5$}, \\
K_{23} &= -\frac{h' H'}{h H},
& K_{24} &= -\frac{\tilde{F}_{tt}}{h^2 \tilde{F}} - \frac{h' \tilde{F}_r}{h\tilde{F}} - \frac{(h')^2}{h^2},
\end{align*}
\begin{align*}
K_{2j} &= -\frac{T''}{h^2 T} - \frac{(h')^2}{h^2} \ \text{for $j \ge 5$},\\
K_{34} &= -\frac{\tilde{F}_r H'}{\tilde{F} H} - \frac{h' H'}{h H},\\
K_{3j} &= -\frac{h' H'}{h H} \ \text{for $j \ge 5$},\\
K_{4j} &= -\frac{\tilde{F}_t T'}{h^2 \tilde{F} T} - \frac{h' \tilde{F}_r}{h \tilde{F}} - \frac{(h')^2}{h^2}
 \ \text{for $j \ge 5$},\\
K_{ij} &= -\frac{(T')^2}{h^2 T^2} - \frac{1}{h^2 T^2} - \frac{(h')^2}{h^2}
\ \text{for $j > i \ge 5$}.
\end{align*}

\begin{lem} \label{lem:key2}
  \begin{enumerate}
  \item $\tilde{F} > 0$, $\tilde{F}_t, \tilde{F}_{tt}, \tilde{F}_r \ge 0$.
  \item $T \ge 1$, $T' \ge 0$, $T'' \ge 0$.
  \item The following functions are all uniformly bounded:
  \[
    \ \frac{\tilde{F}_r}{\tilde{F}}, \ \frac{\tilde{F}_{rr}}{\tilde{F}}, \ \frac{\tilde{F}_t}{\tilde{F}},
    \ \frac{\tilde{F}_{tt}}{\tilde{F}},
    \ \frac{T'}{T}, \ \frac{T''}{T}.
  \]
  \end{enumerate}
\end{lem}

\begin{proof}
(1) follows from the definition of $\tilde{F}$ and Lemma \ref{lem:key}.

(2) is clear.

We prove (3).
The boundedness of $T'/T$ and $T''/T$ are derived from the definition of $T$.
For $t \le 5$, we see that $\tilde{F} = F$ and
the boundedness of $\tilde{F}_r / \tilde{F}$, $\tilde{F}_{rr} / \tilde{F}$,
$\tilde{F}_t / \tilde{F}$, $\tilde{F}_{tt} / \tilde{F}$ follow from Lemma \ref{lem:key}.
For $t \ge 5$, we see that $\tilde{F}$ is independent of $r$, so that
$\tilde{F}_r = \tilde{F}_{rr} = 0$
and that $\tilde{F}_t / \tilde{F}$, $\tilde{F}_{tt} / \tilde{F}$ are bounded for
$t \in [\,5,a\,]$.
For $t \ge a$, we have $\tilde{F} = \sinh(t-5)$, for which
$\tilde{F}_t / \tilde{F}$, $\tilde{F}_{tt} / \tilde{F}$ are bounded
because of $a > 5$.
We thus obtain (3).
This completes the proof of the lemma.
\end{proof}

\begin{lem}\label{lem:K_ij-2}
There is a constant $C<0$ such that 
$C \le K_{ij} <0$ for all $i \neq j$.
\end{lem}

\begin{proof}
The negativity and boundedness of $K_{ij}$
is readily seen from Lemmas \ref{lem:key} and \ref{lem:key2}
except the negativity of $K_{14}$.
We remark that $\tilde{F}_{rr} \ge 0$ does not hold.

In the case where $r \ge 5$, we see that $\tilde{F}$ is independent of $r$
and then
\[
K_{14} = -\frac{h''}{h},
\]
which is negative and bounded by Lemma \ref{lem:key}.

In the case where $r \le 5$, we see $\tilde{F} = F$,
in which case the negativity and the boundedness of $K_{14}$
are proved in the same way as in Lemma \ref{curvature}.
This completes the proof.
\end{proof}

\begin{lem} \label{curvature2}
There is a constant $C<0$ such that 
$C \le K_\sigma <0$ for all $2$-planes $\sigma$
of the tangent spaces at all points.
\end{lem}

\begin{proof}
We prove the lemma in a similar way to that of Lemma \ref{curvature}.

As is already seen in \eqref{eq:R},
for the negativity of $K_\sigma$, it suffices to prove
\begin{equation} \label{eq:R2}
(R_{1442})^2 \le R_{1441} R_{2442}.
\end{equation}
If $t > 5$, then $\tilde{F}$ is independent of $r$ and so
$\tilde{F}_{rt}^2 = \tilde{F}_{rr} \tilde{F}_{tt} = 0$,
which implies \eqref{eq:R2}.
If $t \le 5$, then $\tilde{F} = F$ and the calculation in the proof of
Lemma \ref{curvature} yields \eqref{eq:R2}.
The negativity of $K_\sigma$ follows.

We prove the boundedness of $K_\sigma$ for all $\sigma$.
It suffice to estimate $A_{ijji}$ for $i < j$ and $A_{1442}$,
where $A_{ijkl}$ is defined in the proof of Lemma \ref{curvature}.
By $A_{ijji} \le |K_{ij}|$ and by Lemma \ref{lem:K_ij-2},
we have the boundedness of $A_{ijji}$.
The same calculation as in the proof of Lemma \ref{curvature}
leads us to
\[
A_{1442} \le \frac{\tF_{rt}}{2h\tF} \le \frac{\tF_{rt}}{2\tF}. 
\]
If $t > 5$, then $\tF$ is independent of $r$ and so $\tF_{rt} = 0$.
If $t \le 5$, then $\tF = F = b e^R f(t-R)$ and so
\[
\frac{\tF_{rt}}{2\tF} = \frac{R'(f'-f'')}{f},
\]
which is bounded since $f'-f''$ has compact support.
This completes the proof.
\end{proof}

We are ready to prove Proposition \ref{estimate.2}.
\begin{proof}[Proof of Proposition {\rm\ref{estimate.2}}]
First of all, the rotational invariance of $g$ is clear
by the form of $g$.

(1) By Lemma \ref{curvature2}.

Checking (2) - (6) is similar to (2) - (6) of Proposition \ref{estimate}. We omit it. 

We prove \eqref{it:g-hat-r}.
Assume $r \ge 5$.  The curvature tensor for the metric $\hat{g}$
is obtained as, for $j > i \ge 5$,
\begin{align*}
R_{1331} &= - \tF \tF_{tt}, 
& R_{1441} &= - T T'', \\
R_{1ii1} &= - e^{2w} T T'', 
& R_{3443} &= - \tF \tF_t T T', \\
R_{3ii3} &= - e^{2w} \tF \tF_t T T', 
& R_{4ii4} &= - e^{2w} T^2 (1+(T')^2), \\
R_{ijji} &= - e^{4w} T^2 (1+(T')^2), &
\end{align*}
which are the unique nonzero values of $R_{ijkl}$ under the (skew-)symmetry.
This together with Lemma \ref{lem:key2}(1)(2) implies
the non-positivity of all the sectional curvatures.
We have proved Proposition \ref{estimate.2}.
\end{proof}

\section{Questions}
\subsection{More complicated examples}
As we explained in section \ref{section.construction},
a flip manifold can be obtained as follows: take two surfaces $V_1,V_2$, remove
a small neighborhood of a point $p_i$ from each of them, then consider an $S^1$-bundle
over each. The boundary of each manifold is an $(S^1 \times S^1)$-bundle
over a point ($p_1$ and $p_2$),  and now we glum them by a flip map. 

Regarding the above example, one can view
$V_1$ and $V_2$ are intersecting in one point.
In view of this, a similar construction can be done in dimension $2n, n \ge1$,
with a more complicated intersection pattern. 
The above case
is for $n=1$, and we describe the case for $n=2$.
 Let $V_1$ be a closed hyperbolic $4$-manifold
with two, isometric, totally geodesic embedded closed $2$-submanifolds
$V_{12},V_{13}$ intersecting at one point $V_{123}$ transversally.
Prepare two other copies:
$V_2$ with submanifolds $V_{23},V_{21}$; and $V_3$ with submanifolds $V_{31}, V_{32}$.

Fix a small $\epsilon>0$, and consider an $S^1$-bundle: 
$$X_1 =(V_1 \backslash N_\epsilon(V_{12}\cup V_{13})) \ltimes S^1,$$ whose
boundary is 
$\d N_\epsilon(V_{12}\cup V_{13}) \ltimes S^1$.
Note that  $\d N_\epsilon(V_{12}\cup V_{13})$ is a flip manifold embedded in $V_1$:
$$(V_{12}\backslash N_\epsilon(V_{123}) \ltimes S^1)\cup_{V_{123} \ltimes S^1 \ltimes S^1} (V_{13} \backslash N_\epsilon(V_{123}) \ltimes S^1),$$
where we flip the two $S^1$-fibres in 
$V_{123} \ltimes S^1 \ltimes S^1$ when we glue the left piece to the right one. 
Similarly, consider $S^1$-bundles $X_2, X_3$ for  $V_2, V_3$, respectively. 
We put a locally product metric on each $X_i$.

Now from $X_1, X_2, X_3$, 
we define a $5$-manifold 
$$M^5=(X_1 \cup X_2 \cup X_3)/\sim,$$
where $\sim$ means  gluing among the boundaries of $X_1,X_2,X_3$:
\begin{align*}
  \partial X_1 &= 
  (V_{12} \backslash N_{\epsilon}(V_{123})) \ltimes S^1 \ltimes S^1
\cup_{V_{123}\ltimes S^1 \ltimes S^1
\ltimes S^1}
(V_{13} \backslash N_{\epsilon}(V_{123})) \ltimes S^1 \ltimes S^1, \\
 \partial X_2 &=
 (V_{21} \backslash N_{\epsilon}(V_{123})) \ltimes S^1 \ltimes S^1
\cup_{V_{123}\ltimes S^1 \ltimes S^1
\ltimes S^1}
(V_{23} \backslash N_{\epsilon}(V_{123})) \ltimes S^1 \ltimes S^1, \\
\partial X_3&= 
(V_{32} \backslash N_{\epsilon}(V_{123})) \ltimes S^1 \ltimes S^1
\cup_{V_{123}\ltimes S^1 \ltimes S^1
\ltimes S^1}
(V_{31} \backslash N_{\epsilon}(V_{123})) \ltimes S^1 \ltimes S^1.
 \end{align*}
 
 A gluing map is described as follows for each pair $(i,j)$: use the obvious 
 identification 
 $V_{ij}\backslash N_{\epsilon}(V_{123})=V_{ji}\backslash N_{\epsilon}(V_{123})$ and flip
  the two $S^1$-fibers. 
The common manifold $V_{123}\ltimes S^1 \ltimes S^1
\ltimes S^1$ is shared by all of them in $M$.
We assume that the identification are done by isometries.

It would be interesting to know if $M$
appears as an end (cf. \cite{AS}, see also \cite{B} on the topology
of thsoe ends).
In view of our strategy, 
as the first step we want to know if $M$ has a metric 
of non-positive curvature, but the curvature estimate becomes more subtle
when we look for an eventually warped cusp metric for $\R \times M$.

\subsection{Graph manifolds}
Among graph manifolds $W$, 
we proved that $W$ appears as an end
if it has a Riemannian metric of non-positive curvature
(Corollary \ref{cor.nonpositive}).
See  \cite{BS} on the question to decide which graph manifolds carry Riemannian 
metric of non-positive curvature.
Leeb \cite{L} gave an example of graph manifold that does not have  a metric of 
non-positive curvature.
It would be interesting to know if his examples will/will not appear as an end.

\end{document}